\documentclass[envcountsame]{llncs} 
\usepackage{latexsym} 

\usepackage{times}

\usepackage[dvips]{epsfig} 
\graphicspath{{./Fig_eps/}{./Eps/}} 
\DeclareGraphicsExtensions{.ps,.eps} 
\usepackage{color} 
\usepackage{a4wide}

\title{Local Limit Theorems and Number of Connected Hypergraphs\thanks{The results in this manuscript will appear in {\em Combinatorics, Probability and Computing}  in two papers: {Local limit theorems for the giant component of random hypergraphs} (doi:10.1017/S0963548314000017); {The asymptotic number of connected d-uniform hypergraphs} (doi:10.1017/S0963548314000029)}.}
 
\author{{MICHAEL BEHRISCH}$^1$\thanks{Supported by the DFG research center \textsc{Matheon} in Berlin.},
		{AMIN COJA-OGHLAN}$^2$\thanks{Supported by DFG CO~646.},
		{MIHYUN KANG}$^3$\thanks{Supported by DFG KA 2748/3-1.}
			}

 \institute{		{$^1$}
			Institute of Transportation Systems, German Aerospace Center\\  Rutherfordstrasse 2,  12489 Berlin, Germany \\ 
			\email{michael.behrisch@dlr.de}\vspace{2ex}\\
		{$^2$}
			Goethe University, Mathematics Institute\\ 60054 Frankfurt am Main, Germany\\ 
			\email{acoghlan@math.uni-frankfurt.de}\vspace{2ex}\\
		{$^3$}
			TU Graz, Institut f\"ur Optimierung und Diskrete Mathematik (Math B)\\ Steyrergasse 30,  8010 Graz, Austria\\
			\email{kang@math.tugraz.at}
}
 
\newcommand\myhomepage{http://www.informatik.hu-berlin.de/$\sim$coja/} 

\newcommand{\tsigma}{\tau}

\newcommand\eps{\varepsilon}

\newcommand\ZZ{\bbbz} 
\newcommand\Var{\mathrm{Var}} 
\newcommand\Erw{\mathrm{E}} 
\newcommand\pr{\mathrm{P}}

\newcommand\gnp{G(n,p)} 
\newcommand\gnm{G(n,m)} 
\newcommand\hnm{\hdnm} 
\newcommand\hnp{\hdnp}
\newcommand\hnpz{H_d(n,p_z)} 
\newcommand\hnmz{H_d(n,m_z)} 
\newcommand\hdnm{H_d(n,m)} 
\newcommand\hdnp{H_d(n,p)}

\newcommand{\Bin}{{\rm Bin}} 
\newcommand{\eqref}[1]{(\ref{#1})} 
 
\newcommand\comp{\mathcal{C}}

\newcommand{\bink}[2] 
    {{{#1}\choose {#2}}} 
 
\newcommand\HHH{\mathcal{H}} 
 
\newcommand\III{\mathcal{I}} 
\newcommand\JJJ{\mathcal{J}}

\newcommand\QQQ{\mathcal{Q}} 
 
\newcommand\GGG{G} 
 
\newcommand\bc[1]{\left({#1}\right)} 
 
\newcommand\bcfr[2]{\bc{\frac{#1}{#2}}} 
 
\newcommand\brk[1]{\left\lbrack{#1}\right\rbrack}

\newcommand\abs[1]{\left|{#1}\right|} 
\newcommand\uppergauss[1]{\left\lceil{#1}\right\rceil}

\newcommand\RR{\mathbf{R}}

\newcommand{\whp}{w.h.p.} 
\newcommand{\stacksign}[2]{{\stackrel{\mbox{\scriptsize #1}}{#2}}}

\newcommand\order{\mathcal{N}} 
\newcommand\size{\mathcal{M}}

\newcommand{\Karonski}{Karo\'nski}
\newcommand{\Rucinski}{Ruci\'nski}
\newcommand{\Erdos}{Erd\H{o}s}
\newcommand{\Renyi}{R\'enyi}

\newcommand{\Luczak}{\L uczak}

\newcommand\Lem{Lemma}
\newcommand\Prop{Proposition}
\newcommand\Thm{Theorem}
\newcommand\Cor{Corollary}
\newcommand\Sec{Section}

\newcommand{\cN}{{\mathcal N}}

\newcommand{\fracwd}[2]{\bcfr{#1}{#2}}
\newcommand{\inv}[1]{\frac{1}{#1}}

\newcommand{\binnd}{{\bink{n-1}{d-1}}}

\begin{document} 
 
\maketitle 
 
\spnewtheorem{Algo}[theorem]{Algorithm}{\bfseries}{} 
 
\begin{abstract}
Let $\hnp$ signify a random $d$-uniform hypergraph with $n$ vertices in which
each of the $\bink{n}d$ possible edges is present with probability $p=p(n)$ independently,
and let $\hnm$ denote a uniformly distributed $d$-uniform hypergraph with $n$ vertices and $m$ edges.
We derive local limit theorems for the joint distribution of the number of vertices and the number of edges in the largest component
of $\hnp$ and $\hnm$ in the regime $(d-1)\bink{n-1}{d-1}p>1+\eps$, resp.\ $d(d-1)m/n>1+\eps$,
where $\eps>0$ is arbitrarily small but fixed as $n\rightarrow\infty$. 
As an application, we obtain an asymptotic formula for the probability that $\hnp$ or $\hnm$ is connected.. 
In addition, we derive a local limit theorem for the number of edges in $\hnp$, conditioned on $\hnp$ being connected. While most prior work on this subject relies on techniques from enumerative combinatorics, we present a new, purely probabilistic approach.\vspace{2ex}
\end{abstract}

\section{Introduction and Results}\label{Sec_Intro}

This paper deals with the connected components of random graphs and hypergraphs.
Recall that a \emph{$d$-uniform hypergraph} $H$ is a set $V(H)$ of \emph{vertices} together with a set $E(H)$ of edges
$e\subset V(H)$ of size $|e|=d$.
The \emph{order} of $H$ is the number $|V(H)|$ of vertices of $H$, and the \emph{size} of $H$ is the number $|E(H)|$ of edges.
Moreover, a $2$-uniform hypergraph is called a \emph{graph}.

Further, we say that a vertex $v\in V(H)$ is \emph{reachable} from $w\in V(H)$ if there exists edges $e_1,\ldots,e_k\in E(H)$
such that $v\in e_1$, $w\in e_k$ and $e_i\cap e_{i+1}\not=\emptyset$ for all $1\leq i<k$.
Then reachability is an equivalence relation, and the equivalence classes are called the \emph{components} of $H$.
If $H$ has only a single component, then $H$ is \emph{connected}.

We let $\order(H)$ signify the maximum order of a component of $H$.
Furthermore, for all hypergraphs $H$ the vertex set $V(H)$ will consist of integers.
Therefore, the subsets of $V(H)$ can be ordered lexicographically, and we call the lexicographically first component of $H$
that has order $\order(H)$ the \emph{largest component} of $H$.
In addition, we denote by $\size(H)$ the size of the largest component.

We will consider two models of random $d$-uniform hypergraphs.
The random hypergraph $\hnp$ has the vertex set $V=\{1,\ldots,n\}$, and each of the $\bink{n}d$ possible edges is present
with probability $p$ independently of all others.
Moreover, $\hnm$ is a uniformly distributed hypergraph with vertex set $V=\{1,\ldots,n\}$ and with exactly $m$ edges.
In the case $d=2$, the notation $\gnp=H_2(n,p)$, $\gnm=H_2(n,m)$ is common.
Finally, we say that the random hypergraph $\hnp$ enjoys a certain property $\mathcal{P}$ \emph{with high probability} (``\whp'')
if the probability that $\mathcal{P}$ holds in $\hnp$ tends to $1$ as $n\rightarrow\infty$;
a similar terminology is used for $\hnm$.

\subsection{The Phase Transition and the Giant Component}

In the pioneering papers~\cite{ER1,ER60} on the theory of random graphs,
\Erdos\ and \Renyi\ studied the component structure of the random graph $\gnm$.
Since~\cite{ER1,ER60}, the component structure of random discrete objects (e.g., graphs, hypergraphs, digraphs,~\dots)
has been among the main subjects of discrete probability theory.
One reason for this is the connection to statistical physics and percolation (as ``mean field models'');
another reason is the impact of these considerations on computer science (e.g., due to relations to computational
problems such as \textsc{Max Cut} or \textsc{Max $2$-Sat}~\cite{Sorkin}).

In~\cite{ER1} 
\Erdos\ and \Renyi\ showed that if $t$ remains fixed as $n\rightarrow\infty$ and $m=\frac{n}2(\ln n+t)$, then the probability
that $\gnm$ is connected is asymptotically $\exp(-\exp(-t))$ as $n\rightarrow\infty$.
Since $\gnm$ is a uniformly distributed graph, this result immediately yields the asymptotic number of connected graphs of order $n$ and size $m$.
The relevance of this result,
possibly the most important contribution of~\cite{ER1} is that \Erdos\ and \Renyi\ solved this \emph{enumerative} problem
(``how many connected graphs of order $n$ and size $m$ exist?'') via \emph{probabilistic} methods
(namely, the method of moments for proving convergence to a Poisson distribution).

Furthermore, in~\cite{ER60} \Erdos\ and \Renyi\ went on to study (among others) the component structure of \emph{sparse} random graphs with $O(n)$ edges.
The main result is that the order $\order(\gnm)$ of the largest component undergoes a \emph{phase transition} as $2m/n\sim1$.
Let us state a more general version from~\cite{SPS85}, which covers $d$-uniform hypergraphs:
let either $H=\hnm$ and $c=dm/n$, or $H=\hnp$ and $c=\bink{n-1}{d-1}p$;
we refer to $c$ as the \emph{average degree} of $H$.
Then the result is that
\begin{itemize}
\item[(1)] if $c<(d-1)^{-1}-\eps$ for an arbitrarily small but fixed $\eps>0$, then $\order(H)=O(\ln n)$ \whp\
\item[(2)] By contrast, if $c>(d-1)^{-1}+\eps$, then $H$ features a unique component of order $\Omega(n)$ \whp,
	which is called the \emph{giant component}.
	More precisely, $\order(H)=(1-\rho)n+o(n)$ \whp\, where $\rho$ is the unique solution to the transcendental equation
	\begin{equation}\label{eqCOMV}
	\rho=\exp(c(\rho^{d-1}-1))
	\end{equation}
	that lies strictly between $0$ and $1$.
	Furthermore, the second largest component has order $O(\ln n)$.
\end{itemize}

In this paper we present  a new, \emph{purely probabilistic} approach for investigating the
component structure of sparse random graphs and, more generally, hypergraphs in greater detail.
More precisely, we obtain \emph{local limit theorems} for the joint distribution of the order and size of the largest
component in a random graph or hypergraph $H=\hnm$ or $H=\hnp$.
Thus, we determine the joint limiting distribution of $\order(H)$ and $\size(H)$ precisely.
Furthermore, from these local limit theorems we derive a number of interesting consequences.
For instance, we compute the asymptotic probability that $H$ is connected,
which yields an asymptotic formula for the number of connected hypergraphs of a given order and size.
Thus, as in~\cite{ER1}, we solve a (highly non-trivial) \emph{enumerative} problem via probabilistic techniques.
In addition, we infer a local limit theorem for the distribution of the number of edges of $\hnp$, given
the (exponentially unlikely) event that $\hnp$ is connected.

While in the case of \emph{graphs} (i.e., $d=2$) these results are either known or can be derived from prior
work (in particular, \cite{BCM90}), all our results are new for $d$-uniform hypergraphs with $d>2$.
Besides, we believe that our probabilistic approach is interesting in the case of graphs as well,
because we completely avoid the use of involved enumerative methods, which are the basis of most of
the previous papers on our subject (including~\cite{BCM90}).
In effect, our techniques are fairly generic and may apply to further problems of a related nature.

\subsection{Results}\label{Sec_Results}

\subsubsection{The local limit theorems.}

Our first result is the \emph{local limit theorem} for the joint distribution of $\order(\hnp)$ and $\size(\hnp)$.

\begin{theorem}\label{Thm_Hnplocal}
Let $d\geq2$ be a fixed integer.
For any two compact sets $\III\subset\RR^2$, $\JJJ\subset((d-1)^{-1},\infty)$,
and for any $\delta>0$ there exists $n_0>0$ such that the following holds.
Let $p=p(n)$ be  a sequence such that $c=c(n)=\bink{n-1}{d-1}p\in\JJJ$ for all $n$ and
let $0<\rho=\rho(n)<1$ be the unique solution to (\ref{eqCOMV}).
Further, let
	\begin{eqnarray}\label{eq:defsigmaN}
	\sigma_\order^2&=&\frac{\rho\left(1-\rho + c(d-1)(\rho-\rho^{d-1})\right)}{(1-c(d-1)\rho^{d-1})^2}n,\\
	\sigma_\size^2&=&c^2\rho^d\frac{2 + c(d-1)(\rho^{2d-2} -2\rho^{d-1} +\rho^d) - \rho^{d-1}-\rho^d}{(1-c(d-1)\rho^{d-1})^2}n + (1-\rho^d)\frac{cn}{d},\nonumber\\
	\sigma_{\order\size}&=&c\rho\frac{1-\rho^d - c(d-1)\rho^{d-1}(1-\rho)}{(1-c(d-1)\rho^{d-1})^2}n.\nonumber
	\end{eqnarray}
Suppose that $n\geq n_0$ and that $\nu,\mu$ are integers such that
$x=\nu-(1-\rho)n$ and $y=\mu-(1-\rho^d)\bink{n}dp$ satisfy $n^{-\frac12}\bink{x}{y}\in\III$.
Then letting 
	$$P(x,y)=\inv{2\pi\sqrt{\sigma_\order^2\sigma_\size^2-\sigma_{\order\size}^2}}
		\exp\brk{-\frac{\sigma_\order^2\sigma_\size^2}{2(\sigma_\order^2\sigma_\size^2-\sigma_{\order\size}^2)}
			\left(\frac{x^2}{\sigma_\order^2}-\frac{2\sigma_{\order\size}xy}{\sigma_\order^2\sigma_\size^2}+\frac{y^2}{\sigma_\size^2}\right)}$$
we have
	$(1-\delta)P(x,y) \leq \pr\brk{\order(\hnm)=\nu\wedge\size(\hnm)=\mu} \leq (1+\delta)P(x,y).$
\end{theorem}
\Thm~\ref{Thm_Hnplocal} characterizes the joint limiting distribution of $\order(\hnp)$ and $\size(\hnp)$ \emph{precisely},
because it actually yields the asymptotic probability that $\order$ and $\size$ attain any two values
$\nu=(1-\rho)n+x$, $\mu=(1-\rho^d)\bink{n}dp+y$;
namely, the theorem shows that
	$$
	\pr\brk{\order(\hnp)=\nu\wedge\size(\hnp)=\mu}\sim P(x,y),
	$$
and it guarantees some uniformity of convergence.
We emphasize that $P(x,y)$ is as small as $O(n^{-1})$ as $n\rightarrow\infty$.
Since $P(x,y)$ is just the density function of a bivariate normal distribution,
\Thm~\ref{Thm_Hnplocal} readily yields the following \emph{central limit theorem} for the joint distribution of $\order,\size(\hnp)$.

\begin{corollary}\label{Cor_Hnplocal}
With the notation and the assumptions of \Thm~\ref{Thm_Hnplocal},
suppose that the limit $\Xi=\lim_{n\rightarrow\infty}\frac{\sigma_{\order\size}}{\sigma_\order\sigma_\size}$
exists.
Then the joint distribution of
	$$\frac{\order(\hnp)-(1-\rho)n}{\sigma_\order}\quad\mbox{ and }\quad\frac{\size(\hnp)-(1-\rho^d)\bink{n}dp}{\sigma_\size}$$
converges in distribution to the bivariate normal distribution with mean $0$ and covariance matrix
	$\bc{\begin{array}{cc}1&\Xi\\\Xi&1\end{array}}.$
\end{corollary}
Nonetheless, we stress that \Thm~\ref{Thm_Hnplocal} is considerably more precise than \Cor~\ref{Cor_Hnplocal}.
For the latter result just yields the asymptotic probability that
$x\leq\sigma_\order^{-1}(\order(\hnp)-(1-\rho)n\leq x'$ and simultaneously
$y\leq\sigma_\size^{-1}(\size(\hnp)-(1-\rho^d)n)\leq y'$ for any fixed $x,x',y,y'\in\RR$.
Hence, \Cor~\ref{Cor_Hnplocal} just determines $\order,\size(\hnp)$ up to errors of $o(\sigma_\order)$ and $o(\sigma_\size)$,
while \Thm~\ref{Thm_Hnplocal} actually yields the probability of hitting \emph{exactly} specific values $\nu,\mu$.

The second main result of this paper is a local limit theorem for the joint distribution of
$\order(\hnm)$ and $\size(\hnm)$.

\begin{theorem}\label{Thm_Hnmlocal}
Let $d\geq2$ be a fixed integer.
For any two compact sets $\III\subset\RR^2$, $\JJJ\subset((d-1)^{-1},\infty)$,
and for any $\delta>0$ there exists $n_0>0$ such that the following holds.
Let $m=m(n)$ be  a sequence of integers such that $c=c(n)=dm/n\in\JJJ$ for all $n$ and
let $0<\rho=\rho(n)<1$ be the unique solution to (\ref{eqCOMV}).
Further, let
	\begin{eqnarray*}
	\tsigma_\order^2&=&\rho\frac{1-(c+1)\rho -c(d-1)\rho^{d-1} + 2cd\rho^d - cd\rho^{2d-1}}{(1-c(d-1)\rho^{d-1})^2}n,\\
	\tsigma_\size^2&=&c\rho^d\frac{1 - c(d-2)\rho^{d-1} - (c^2d-cd+1)\rho^d - c^2(d-1)\rho^{2d-2} + 2c(cd-1)\rho^{2d-1} - c^2\rho^{3d-2}}{d(1-c(d-1)\rho^{d-1})^2}n,\\
	\tsigma_{\order\size}&=&c\rho^d\frac{1-c\rho - c(d-1)\rho^{d-1} + (c+cd-1)\rho^d - c \rho^{2d-1}}{(1-c(d-1)\rho^{d-1})^2}n.
	\end{eqnarray*}
Suppose that $n\geq n_0$ and that $\nu,\mu$ are integers such that $x=\nu-(1-\rho)n$ and $y=\mu-(1-\rho^d)m$ satisfy $n^{-\frac12}\bink{x}{y}\in\III$.
Then letting 
	$$Q(x,y)=\inv{2\pi\sqrt{\tsigma_\order^2\tsigma_\size^2-\tsigma_{\order\size}^2}}\exp\brk{-\frac{\tsigma_\order^2\tsigma_\size^2}{2(\tsigma_\order^2\tsigma_\size^2-\tsigma_{\order\size}^2)}\left(\frac{x^2}{\tsigma_\order^2}-\frac{2\tsigma_{\order\size}xy}{\tsigma_\order^2\tsigma_\size^2}+\frac{y^2}{\tsigma_\size^2}\right)}$$
we have
	$(1-\delta)Q(x,y) \leq \pr\brk{\order(\hnm)=\nu\wedge\size(\hnm)=\mu} \leq (1+\delta)Q(x,y).$
\end{theorem}
Similarly as \Thm~\ref{Thm_Hnplocal}, \Thm~\ref{Thm_Hnmlocal} characterizes the joint limiting distribution
of $\order,\size(\hnm)$ precisely.
Once more the limit resembles a bivariate normal distribution, so that we can infer the following
central limit theorem.

\begin{corollary}\label{Cor_Hnmlocal}
With the notation and the assumptions of~\Thm~\ref{Thm_Hnmlocal},
suppose that the limit $\Xi=\lim_{n\rightarrow\infty}\frac{\tau_{\order\size}}{\tau_\order\tau_\size}$ exists.
Then the joint distribution of
	$$\frac{\order(\hnm)-(1-\rho)n}{\tau_\order}\quad\mbox{and}\quad\frac{\size(\hnm)-(1-\rho^d)m}{\tau_\size}$$
converges in distribution to the bivariate normal distribution with mean $0$ and covariance matrix
	$\bc{\begin{array}{cc}1&\Xi\\\Xi&1\end{array}}$.
\end{corollary}

\subsubsection{The probability of connectedness.}
As an application of the local limit theorem for $\hnp$ (\Thm~\ref{Thm_Hnplocal}), we obtain
the following formula for the asymptotic probability that a random hypergraph $H_d(\nu,\mu)$ is connected,
and thus for the number of connected hypergraphs of a given order and size.

\begin{theorem}\label{Thm_Cnm}\label{thm:cnm}
Let $d\geq2$ be a fixed integer.
For any compact set $\JJJ\subset(d(d-1)^{-1},\infty)$,
and for any $\delta>0$ there exists $\nu_0>0$ such that the following holds.
Let $\mu=\mu(\nu)$ be  a sequence of integers such that $\zeta=\zeta(\nu)=d\mu/\nu\in\JJJ$ for all $\nu$.
Then the there exists a unique number $0<r=r(\nu)<1$ such that
\begin{equation}\label{eq:hnp}
r=\exp\bc{-\zeta\cdot\frac{(1-r)(1-r^{d-1})}{1-r^d}}.
\end{equation}
Let $\Phi_d(r,\zeta)=r^\frac{r}{1-r}(1-r)^{1-\zeta}(1-r^d)^\frac{\zeta}{d}$ for $d\ge 2$.
Furthermore, define, for $d> 2$, 
\begin{eqnarray*}
R_d(\nu,\mu)
&=&\frac{1-r^d- (1-r)(d-1)\zeta r^{d-1}}{\sqrt{\big(1-r^d+ \zeta(d-1)(r-r^{d-1})\big)(1-r^d) -d \zeta r(1-r^{d-1})^2}}\nonumber\\
&&\cdot \exp\bc{\frac{(d-1)\zeta(r -2r^d+r^{d-1})}{2(1-r^d)}}\cdot\Phi_d(r,\zeta)^{\nu},
\end{eqnarray*}
and for $d=2$, 
\begin{eqnarray*}
R_2(\nu,\mu)&=&\frac{1+r- \zeta r}{\sqrt{\bc{1+r}^2 -2\zeta r}}\cdot \exp\bc{\frac{2\zeta r +\zeta^2r}{2(1+r)}}\cdot\Phi_2(r,\zeta)^{\nu}.
\end{eqnarray*}

Finally, let $c_d(\nu,\mu)$ denote the probability that $H_d(\nu,\mu)$ is connected.
Then for all $\nu>\nu_0$ we have 
	$$(1-\delta)R_d(\nu,\mu)<c_d(\nu,\mu)<(1+\delta)R_d(\nu,\mu).$$
\end{theorem}
To prove \Thm~\ref{Thm_Cnm} we
shall consider a ``larger'' hypergraph $\hnp$ such that the
expected order and size of the largest component of $\hnp$ are $\nu$ and $\mu$.
Then, we will infer the probability that $H_d(\nu,\mu)$ is conncected from the local limit theorem
for $\order(\hnp)$ and $\size(\hnp)$.
Indeed, this proof needs the full strength of \Thm~\ref{Thm_Hnplocal};
i.e., the central limit theorem provided by \Cor~\ref{Cor_Hnplocal} is insufficient to obtain \Thm~\ref{thm:cnm}.

Furthermore, we have the following theorem on the asymptotic probability that $H_d(\nu,p)$ is connected.

\begin{theorem}\label{Thm_Cnp}\label{thm:cd}
Let $d\geq2$ be a fixed integer.
For any compact set $\JJJ\subset(0,\infty)$,
and for any $\delta>0$ there exists $\nu_0>0$ such that the following holds.
Let $p=p(\nu)$ be  a sequence such that $\zeta=\zeta(\nu)=\bink{\nu-1}{d-1}p\in\JJJ$ for all $\nu$.
Then there is a unique $0<\varrho=\varrho(\nu)<1$ such that
\begin{equation}\label{eq:varrho}
\varrho=\exp\brk{\zeta\frac{\varrho^{d-1}-1}{(1-\varrho)^{d-1}}}.
\end{equation}

Let $\Psi_d(\varrho,\zeta)=(1-\varrho)\varrho^{\frac{\varrho}{1-\varrho}}
		\exp\left(\frac{\zeta}{d}\cdot\frac{1-\varrho^d-(1-\varrho)^d}{(1-\varrho)^d}\right)$  for $d\ge 2$.
Define, for $d>2$,
\begin{eqnarray*}
S_d(\nu,p) 
&=& \frac{1-\zeta(d-1)\fracwd{\varrho}{1-\varrho}^{d-1}}{\sqrt{1 + \zeta(d-1)\frac{\varrho-\varrho^{d-1}}{(1-\varrho)^d}}} \cdot \exp\bc{\frac{\zeta(d-1)\varrho(1-\varrho^d-(1-\varrho)^d)}{2(1-\varrho)^d}}\\
&& \cdot \exp\bc{\frac{\zeta(d-1)\varrho}2\bc{\bcfr{\varrho}{1-\varrho}^{d-2}+1}} \cdot \Psi_d(\varrho,\zeta)^{\nu},
\end{eqnarray*}
and for $d=2$, 
\begin{eqnarray*}
S_2(\nu,p)&=&
\bc{1-\frac{\zeta}{e^{\zeta}-1}}\cdot \exp\bc{\frac{\zeta (2 +\zeta)}{2(e^{\zeta}-1)}} \cdot  (1-e^{-\zeta})^{\nu}.
\end{eqnarray*}
Finally, let $c_d(\nu,p)$ denote the probability that $H_d(\nu,p)$ is connected. 
Then for all $\nu>\nu_0$ we have
	$$(1-\delta)S_d(\nu,p)<c_d(\nu,p)<(1+\delta)S_d(\nu,p).$$
\end{theorem}

\begin{remark}
The formula for $S_d(\nu,p)$ for $d\ge 2$ given in an extended abstract version~\cite{BCKrandom} of this work was incorrect.
\end{remark}

Interestingly, if we choose $p=p(\nu)$ and $\mu=\mu(\nu)$ in such a way that $\bink{\nu}dp=\mu$ for all $\nu$
and set $\zeta=\bink{\nu-1}{d-1}p=d\mu/\nu$,
then the function $\Psi(\zeta)$ from \Thm~\ref{thm:cd} is strictly bigger than $\Phi(\zeta)$ from \Thm~\ref{thm:cnm}.
Consequently, the probability that $\hnp$ is connected exceeds the probability that $\hnm$ is connected by
an exponential factor.
The reason is that in $\hnp$ the total number of edges is a (bionmially distributed) random variable.
Roughly speaking, $\hnp$ ``boosts'' its probability of connectivity by including a number of edges that exceeds
$\bink{n}dp$ considerably.
That is, the total number of edges \emph{given that $\hnp$ is connected} is significantly bigger than $\bink{n}dp$.

\subsubsection{The distribution of $\size(\hnp)$ given connectivity.}
The following local limit theorem for the total number of edges in $\hnp$ given 
that $\hnp$ is connected quantifies this observation.

\begin{theorem}\label{Thm_edges}\label{thm:edgedist}
Let $d\geq2$ be a fixed integer.
For any two compact sets $\III\subset\RR$, $\JJJ\subset(0,\infty)$,
and for any $\delta>0$ there exists $\nu_0>0$ such that the following holds.
Suppose that $0<p=p(\nu)<1$ is sequence such that $\zeta=\zeta(\nu)=\bink{\nu-1}{d-1}p\in\JJJ$ for all $\nu$.
Moreover, let $0<\varrho=\varrho(\nu)<1$ be the unique solution to \eqref{eq:varrho}, and set
	$$
	\hat\mu= \uppergauss{\frac{\zeta\nu}{d}\cdot\frac{1-\varrho^d}{(1-\varrho)^d}},\ 
	\hat\sigma^2=\frac{\zeta\nu}{d(1-\varrho)^d}\brk{1-\frac{\zeta d\varrho(1-\varrho^{d-1})^2}{(1-\varrho)^d + \zeta(d-1)(\varrho-\varrho^{d-1})}-\varrho^d}.
	$$
Then for all $\nu\geq \nu_0$ and all integers $y$ such that $y\nu^{-\frac12}\in\III$ we have
\begin{eqnarray*}
\frac{1-\delta}{\sqrt{2\pi}\hat\sigma}\exp\left(-\frac{y^2}{2\hat\sigma^2}\right)&\leq&
\pr\brk{\abs{E(H_d(\nu,p))} = \hat\mu + y\;|\;H_d(\nu,p)\textrm{ is connected}}
\leq\frac{1+\delta}{\sqrt{2\pi}\hat\sigma}\exp\left(-\frac{y^2}{2\hat\sigma^2}\right).
\end{eqnarray*}
\end{theorem}
In the case $d=2$ the unique solution to~(\ref{eq:varrho}) is $\varrho=\exp(-\zeta)$, whence the formulas from \Thm~\ref{Thm_edges} simplify to
$\hat\mu=\uppergauss{\frac{\zeta\nu}2\mathrm{coth}(\zeta/2)}$ and
$\hat\sigma^2=\frac{\zeta\nu}2\cdot\frac{1-2\zeta\exp(-\zeta)-\exp(-2\zeta)}{(1-\exp(-\zeta))^2}$.

\subsection{Related Work}\label{Sec_Related}

\subsubsection{Graphs.}
Bender, Canfield, and McKay~\cite{BCM90} were the first to compute the asymptotic probability that a random graph $\gnm$
is connected for \emph{any} ratio $m/n$.
Although they employ a probabilistic result from \Luczak~\cite{Luc90} to simplify their arguments,
their proof is based on enumerative methods.
Additionally, using their formula for the connectivity probability of $\gnm$, Bender, Canfield, and McKay~\cite{BCM92} inferred
the probability that $\gnp$ is connected as well as a central limit theorem for the number of edges of $\gnp$ given connectivity.
Moreover, it is possible (though somewhat technical) to derive local limit theorems for $\order,\size(\gnm)$ and
$\order,\size(\gnp)$ from the main result of~\cite{BCM90}. 
In fact, Pittel and Wormald~\cite{PW,PW05} recently used enumerative arguments to derive an improved version of the
main result of~\cite{BCM90} and to obtain a local limit theorem that in
addition to $\order$ and $\size$ also includes the order and size of the $2$-core.
In summary, in~\cite{BCM90,BCM92,PW,PW05} enumerative results on the number of connected graphs of a given order
and size were used to infer the distributions of $\order,\size(\gnp)$ and $\order,\size(\gnm)$.
By contrast, in the present paper we use the converse approach: employing probabilistic methods,
we first determine the distributions of $\order,\size(\gnp)$ and $\order,\size(\gnm)$, and from this we derive
the number of connected graphs with given order and size.

The asymptotic probability that $\gnp$ is connected was first computed by Stepanov~\cite{Ste70}
(this problem is significantly simpler than computing the probability that $\gnm$ is connected).
He also obtained a local limit theorem for $\order(\gnp)$ (but his methods do not yield the joint distribution of $\order(\gnp)$ and $\size(\gnp)$).
Moreover, 
Pittel~\cite{P90} derived central limit theorems for $\order(\gnp)$ and $\order(\gnm)$
from his result on the joint distribution of the numbers of trees of given sizes outside the giant component.
The arguments in both~\cite{P90,Ste70} are of enumerative and analytic nature.

Furthermore, a few authors have applied probabilistic arguments to problems related to the present work.
For instance, O'Connell~\cite{OCo98} employed the theory of large deviations in order to estimate the probability
that $\gnp$ is connected up to a factor $\exp(o(n))$.
While this result is significantly less precise than Stepanov's, O'Connell's proof is simpler.
In addition, Barraez, Boucheron, and Fernandez de la Vega~\cite{BBF00} exploited the analogy between the
component structure of $\gnp$ and branching processes to derive a central limit theorem
for the joint distribution of $\order(\gnp)$ and the \emph{total} number of edges in $\gnp$;
however, their techniques do not yield a \emph{local} limit theorem.
Finally, Spencer and van der Hofstad~\cite{HS05} used a novel perspective on the branching process
argument to rederive the formula of Bender, Canfield, and McKay~\cite{BCM90} for the number of connected graphs.

\subsubsection{Hypergraphs.}
In contrast to the case of graphs ($d=2$), little is known about the phase transition and the connectivity probability
of random $d$-uniform hypergraphs with $d>2$.
In fact, to our knowledge the arguments used in all of the aforementioned papers 
do not extend to the case $d>2$.

\Karonski\ and \Luczak~\cite{KL97} derived an asymptotic formula for the number of connected
$d$-uniform hypergraphs of order $n$ and size $m=\frac{n}{d-1}+o(\ln n/\ln\ln n)$ via combinatorial techniques.
Since the minimum number of edges necessary for connectivity is $\frac{n-1}{d-1}$, this formula addresses \emph{sparsely} connected
hypergraphs.
Using this result, \Karonski\ and \Luczak~\cite{KL02} investigated the phase transition in $\hnm$ and $\hnp$.
They obtained local limit theorems for the joint distribution of $\order,\size(\hnm)$ and $\order,\size(\hnp)$ in the \emph{early supercritical phase},
i.e., their result apply to the case $m=\bink{n}dp=\frac{n}{d(d-1)}+o(n^{2/3}(\ln n/\ln\ln n)^{1/3})$.
Furthermore, Andriamampianina and Ravelomanana~\cite{AR05} extended the result from \cite{KL97}
to the regime $m=\frac{n}{d-1}+o(n^{1/3})$ via enumerative techniques.
In addition, relying on~\cite{AR05}, Ravelomanana and Rijamamy~\cite{RR} extended~\cite{KL02}
to $m=\bink{n}dp=\frac{n}{d(d-1)}+o(n^{7/9})$.
Note that all of these results either deal with \emph{very sparsely} connected hypergraphs (i.e., $m=\frac{n}{d-1}+o(n)$),
or with the \emph{early} supercritical phase (i.e., $m=\bink{n}dp=\frac{n}{d(d-1)}+o(n)$).
By contrast, the results of this paper concern connected hypergraphs with $m=\frac{n}{d-1}+\Omega(n)$ edges
and the component structure of random hypergraphs $\hnm$ or $\hnp$ with $m=\bink{n}dp=\frac{n}{d(d-1)}+\Omega(n)$.
Thus, our results and those from~\cite{AR05,KL97,KL02,RR} are complementary.

The regime of $m$ and $p$ that we deal with in the present work was previously studied
by Coja-Oghlan, Moore, and Sanwalani~\cite{CMS04} using probabilistic arguments.
Setting up an analogy between a certain branching process and the component structure of $\hnp$,
they computed the expected order and size of the largest component of $\hnp$
along with the variance of $\order(\hnp)$.
Furthermore, they computed the probability that $\hnm$ or $\hnp$ is connected \emph{up to a constant factor},
and estimated the \emph{expected} number of edges of $\hnp$ given connectivity.
Note that \Thm s~\ref{Thm_Cnm}, \ref{Thm_Cnp}, and~\ref{Thm_edges} enhance these results considerably,
as they yield \emph{tight} asymptotics for the connectivity probability, respectively the precise limiting
distribution of the number of edges given conncectivity.

While the arguments of~\cite{CMS04} by themselves are not strong enough to yield local limit theorems,
combining the branching process arguments with further probabilistic techniques,
in~\cite{BCK06} we inferred a local limit theorem for $\order(\hnp)$.
\Thm s~\ref{Thm_Hnplocal} and~\ref{Thm_Hnmlocal} extend this result by giving local limit theorems for the \emph{joint} distribution
of $\order$ and $\size$.

\subsection{Techniques and Outline}\label{Sec_TechOutline}

To prove \Thm s~\ref{Thm_Hnplocal} and~\ref{Thm_Hnmlocal}, we build upon a qualitative result on
the connected components of $\hnp$ from Coja-Oghlan, Moore, and Sanwalani~\cite{CMS04}, and
a local limit theorem for $\order(\hnp)$ from our previous paper~\cite{BCK06}
(\Thm s~\ref{Thm_global}  and~\ref{Thm_Nlocal}, cf.~\Sec~\ref{Sec_Pre}).
The proofs of both of these ingredients solely rely on probabilistic reasoning (namely, branching processes
and Stein's method for proving convergence to a Gaussian).

In \Sec~\ref{Sec_NMlocal} we show that (somewhat surprisingly) the \emph{univariate} local limit theorem for
$\order(\hnp)$ can be converted into a \emph{bivariate} local limit theorem for $\order(\hnm)$ and $\size(\hnm)$.
To this end, we observe that the local limit theorem for $\order(\hnp)$ implies a bivariate local limit theorem
for the joint distribution of $\order(\hnp)$ and the number $\bar\size(\hnp)$ of edges \emph{outside} the largest component.
Then, we will set up a relationship between the joint distribution of $\order,\bar\size(\hnp)$ and that of $\order,\bar\size(\hnm)$.
Since we already know the distribution of $\order,\bar\size(\hnp)$,
we will be able to infer the joint distribution of $\order,\bar\size(\hnm)$ via Fourier analysis.
As in $\hnm$ the \emph{total} number of edges is fixed (namely, $m$), we have $\bar\size(\hnm)=m-\size(\hnm)$.
Hence, we obtain a local limit theorem for the joint distribution of $\order,\size(\hnm)$, i.e., \Thm~\ref{Thm_Hnmlocal}.
Finally, \Thm~\ref{Thm_Hnmlocal} easily implies \Thm~\ref{Thm_Hnplocal}.
We consider this Fourier analytic approach for proving the bivariate local limit theorems the main contribution of the present work.

Furthermore, in \Sec~\ref{sect:cnm} we derive \Thm~\ref{Thm_Cnm} from \Thm~\ref{Thm_Hnplocal}.
The basic reason why this is possible is that \emph{given} that the largest component of $\hnp$ has order $\nu$ and size $\mu$,
this component is a uniformly distributed random hypergraph with these parameters.
Indeed, this observation was also exploited by \Luczak~\cite{Luc90} to estimate the number of connected
graphs up to a polynomial factor, and in~\cite{CMS04}, where an explicit relation between
the numbers $C_d(\nu,\mu)$ and $\pr\brk{\order(\hnp)=\nu\wedge\size(\hnp)=\mu}$ was derived (cf.\ \Lem~\ref{Lemma_CnmAux} below).
Combining this relation with \Thm~\ref{Thm_Hnplocal}, we obtain \Thm~\ref{Thm_Cnm}.
Moreover, in \Sec s~\ref{sect:proofc} and~\ref{sect:edgedist} we use similar arguments to establish \Thm s~\ref{Thm_Cnp} and~\ref{Thm_edges}.

\section{Preliminaries}\label{Sec_Pre}

We shall make repeated use of the following \emph{Chernoff bound} on the tails of a binomially distributed variable $X=\Bin(\nu,q)$
(cf.~\cite[p.~26]{JLR00} for a proof):
for any $t>0$ we have
	\begin{equation}\label{eqChernoff}
	\pr\brk{\abs{X-\Erw(X)}\geq t}\leq2\exp\bc{-\frac{t^2}{2(\Erw(X)+t/3)}}.
	\end{equation}
Moreover, we employ the following \emph{local limit theorem} for the binomial distribution (e.g.,~\cite[Chapter~1]{BB}).
\begin{proposition}\label{Prop_Bin}
Suppose that $0\leq p=p(n)\leq1$ is a sequence such that $np(1-p)\rightarrow\infty$ as $n\rightarrow\infty$.
Let $X=\Bin(n,p)$.
Then for any sequence $x=x(n)$ of integers such that $|x-np|=o(np(1-p))^{2/3}$,
	$$\pr\brk{X=x}\sim\frac{1}{\sqrt{2\pi np(1-p)}}\exp\bc{-\frac{(x-np)^2}{2p(1-p)n}}\qquad
		\mbox{as }n\rightarrow\infty.$$
\end{proposition}

Furthermore, we make use of the following theorem,
which summarizes results from~\cite[\Sec~6]{CMS04} on the component structure of $\hnp$.
\begin{theorem}\label{Thm_global}
Let $c=c(n)$ be a sequence of non-negative reals and let
$p=c\bink{n-1}{d-1}^{-1}$ and $m=\bink{n}dp=cn/d$.
Then for both $H=\hnp$ and $H=\hnm$ the following holds.
\begin{enumerate}
\item[(i)] For any $c_0<(d-1)^{-1}$ there is a number $n_0$ such that
		for all $n>n_0$ for which $c=c(n)\leq c_0$ we have
		$$\pr\brk{\order(H)\leq300(d-1)^2(1-(d-1)c_0)^{-2}\ln n}\geq1-n^{-100}.$$
\item[(ii)] For any $c_0>(d-1)^{-1}$ there are numbers $n_0>0$, $0<c_0'<(d-1)^{-1}$ such that
		for all $n>n_0$ for which $c_0\leq c=c(n)<\ln n/\ln\ln n$ the following holds.
	The transcendental equation~(\ref{eqCOMV}) has a unique solution $0<\rho=\rho(n)<1$,
	which satisfies
		\begin{eqnarray*}
		\rho^{d-1}c<c_0'.
		\end{eqnarray*}
	Furthermore, with probability $\geq1-n^{-100}$ there exists precisely one
	component of order $(1-\rho)n+o(n)$ in $H$, while all other components have order $\leq\ln^2 n$.
	In addition,
		$$\Erw\brk{\order(H)}=(1-\rho)n+o(\sqrt n).$$
\end{enumerate}
\end{theorem} 

Finally, we need the following local limit theorem for $\order(\hnp)$ from~\cite{BCK06}.

\begin{theorem}\label{Thm_Nlocal}
Let $d\geq2$ be a fixed integer.
For any two compact intervals $\III\subset\RR$, $\JJJ\subset((d-1)^{-1},\infty)$,
and for any $\delta>0$ there exist $n_0>0$ and $C_0>0$ such that the following holds.
Let $p=p(n)$ be  a sequence such that $c=c(n)=\bink{n-1}{d-1}p\in\JJJ$ for all $n$.
Then for all $n\geq n_0$ the following two statements are true.
\begin{enumerate}
\item[(i)] We have $\pr\brk{\order(\hnp)=\nu}\leq C_0/\sqrt{n}$ for all $\nu$.
\item[(ii)] Let $0<\rho=\rho(n)<1$ be the unique solution to (\ref{eqCOMV}), and let $\sigma_\order$ be as in~(\ref{eq:defsigmaN}).
	If $\nu$ is an integer such that $\sigma_\order^{-1}(\nu-(1-\rho)n)\in\III$, then
		\begin{eqnarray*}
		\frac{1-\delta}{\sqrt{2\pi}\sigma_\order}\exp\brk{-\frac{(\nu-(1-\rho)n)^2}{2\sigma_\order^2}}&\leq&
			\pr\brk{\order(\hnp)=\nu}\\
			&\leq&\frac{1+\delta}{\sqrt{2\pi}\sigma_\order}\exp\brk{-\frac{(\nu-(1-\rho)n)^2}{2\sigma_\order^2}}.
		\end{eqnarray*}
\end{enumerate}
\end{theorem}

\section{The Local Limit Theorems: Proofs of \Thm s~\ref{Thm_Hnplocal} and~\ref{Thm_Hnmlocal}}\label{Sec_NMlocal}

Throughout this section, we let $\JJJ\subset((d-1)^{-1},\infty)$ and $\III\subset\RR^2$ denote compact sets.
Moreover, we let $\delta>0$ be arbitrarily small but fixed.
In addition, $0<p=p(n)<1$ is a sequence of edge probabilities such that $\bink{n-1}{d-1}p\in\JJJ$ for all $n$.
Then by \Thm~\ref{Thm_global} there exists a unique $0<\rho=\rho(n)<1$ such that $\rho=\exp(\bink{n-1}{d-1}p(\rho^{d-1}-1))$.
Moreover, we let $\sigma=\sqrt{\bink{n}dp(1-p)}$.

Furthermore, we consider two sequences $\nu=\nu(n)$ and $\bar\mu=\bar\mu(n)$ of integers.
We set 
$$x=x(n)=\nu-(1-\rho)n \quad \textit{and } \quad y=y(n)=\rho^d\bink{n}dp-\bar\mu.$$
We assume that $|x|,|y|\leq\sqrt{n}\ln n$.

\subsection{Outline}

In order to prove \Thm~\ref{Thm_Hnmlocal},
our starting point is \Thm~\ref{Thm_Nlocal}, i.e., the local limit theorem for $\order(\hnp)$;
we shall convert this \emph{univariate} limit theorem into a \emph{bivariate} one that covers both $\order$ and $\size$.
To this end, we observe that \Thm~\ref{Thm_Nlocal} easily yields a local limit theorem for the joint distribution of $\order(\hnp)$
and the number $\bar\size(\hnp)$ of edges \emph{outside} the largest component of $\hnp$.
Indeed, we shall prove that \emph{given} that $\order(\hnp)=\nu$, the random variable $\bar\size(\hnp)$ has
approximately a binomial distribution $\Bin(\bink{n-\nu}d,p)$. 
That is, 
	\begin{equation}\label{eqBivOut1}
	\pr\brk{\order(\hnp)=\nu\wedge\bar\size(\hnp)=\bar\mu}
		\sim\pr\brk{\order(\hnp)=\nu}\cdot\pr\brk{\Bin\bc{\bink{n-\nu}d,p}=\bar\mu}.
	\end{equation}
As \Thm~\ref{Thm_Nlocal} and \Prop~\ref{Prop_Bin} yield explicit formulas for the two factors on the r.h.s.,
we can thus infer an explicit formula for  $\pr\brk{\order(\hnp)=\nu\wedge\bar\size(\hnp)=\bar\mu}$. 
However, this does \emph{not} yield a result on the joint distribution of
$\order(\hnp)$ and $\size(\hnp)$.
For the random variables $\size(\hnp)$ and $\bar\size(\hnp)$ are not directly related, because
the \emph{total} number of edges in $\hnp$ is a random variable.

Therefore, to derive the joint distribution of $\order(\hnp)$ and $\size(\hnp)$, we make a detour to the $\hnm$ model,
in which the total number of edges is fixed (namely, $m$).
Hence, in $\hnm$ the step from $\size$ to $\bar\size$ is easy (because $\bar\size(\hnm)=m-\size(\hnm)$).
Moreover, $\hnp$ and $\hnm$ are related as follows:
given that the total number of edges in $\hnp$ equals $m$, $\hnp$ is distributed as $\hnm$.
Consequently,
	\begin{eqnarray}\nonumber
	\pr\brk{\order(\hnp)=\nu\wedge\bar\size(\hnp)=\bar\mu}&=&\\
		&\hspace{-8cm}=&\hspace{-4cm}\;
			\sum_{m=0}^{\bink{n}d}\pr\brk{\Bin\bc{\bink{n}d,p}=m}\cdot\pr\brk{\order(\hnm)=\nu\wedge\bar\size(\hnm)=\bar\mu}.
		\label{eqBivOut2}
	\end{eqnarray}
	
As a next step, we would like to ``solve'' (\ref{eqBivOut2}) for $\pr\brk{\order(\hnm)=\nu\wedge\bar\size(\hnm)=\bar\mu}$.
To this end, recall that~(\ref{eqBivOut1}) yields an explicit expression for the l.h.s.\ of~(\ref{eqBivOut2}).
Moreover, \Prop~\ref{Prop_Bin} provides an explicit formula for the second factor on the r.h.s.\ of~(\ref{eqBivOut2}).
Now, the crucial observation is that the terms $\pr\brk{\order(\hnm)=\nu\wedge\bar\size(\hnm)=\bar\mu}$
we are after are \emph{independent of $p$}, while~(\ref{eqBivOut2}) is true \emph{for all $p$}.

To exploit this observation, let 
$$p_z=p+z\sigma\bink{n}d^{-1}\quad \textit{and} \quad m_z=\lceil\bink{n}dp_z\rceil=\lceil\bink{n}dp+z\sigma\rceil,$$ and set $z^*=\ln^2n$.
Moreover, consider the two functions
	\begin{eqnarray*}
	f(z)=f_{n,\nu,\mu}(z)&=&\left\{\begin{array}{cl}
		n\pr\brk{\order(\hnpz)=\nu\wedge\bar\size(\hnpz)=\bar\mu}&\mbox{ if }z\in\brk{-z^*,z^*}\\
		0&\mbox{ if }z\in\RR\setminus\brk{-z^*,z^*},
	\end{array}\right.\\
	g(z)=g_{n,\nu,\mu}(z)&=&\left\{\begin{array}{cl}n\pr\brk{\order(\hnmz)=\nu\wedge\bar\size(\hnmz)=\bar\mu}&\mbox{ if }z\in\brk{-z^*,z^*}\\
		0&\mbox{ if }z\in\RR\setminus\brk{-z^*,z^*}.\end{array}\right.
	\end{eqnarray*}
Then computing the coefficients $\pr\brk{\order(\hnm)=\nu\wedge\bar\size(\hnm)=\bar\mu}$ is the same as computing the function $g$ explicitly.
To this end, 
we are going to show that (\ref{eqBivOut2}) can be restated as $\|f-g*\phi\|_2=o(1)$. 
Further, this relation in combination with some Fourier analysis will yield a formula for $g(z)$.
Although $f(z)$ and $g(z)$ depend on $n$ and on $\nu=\nu(n)$ and $\mu=\mu(n)$,
in the sequel we will omit these indices to ease up the notation, while keeping in mind that actually $f(z)$ and $g(z)$
represent sequences of functions.

To see that~(\ref{eqBivOut2}) implies $\|f-g*\phi\|_2=o(1)$, we need to analyze some properties of the functions $f$ and $g$.
Using \Thm~\ref{Thm_Nlocal} and \Prop~\ref{Prop_Bin}, we can estimate $f$ as follows.

\begin{lemma}\label{Lemma_fexplicit}
There exists a number $\gamma_0>0$ 
such that
for each $\gamma>\gamma_0$ there exists $n_0>0$ so that for all $n\geq n_0$ the following holds.
\begin{enumerate}
\item We have $f(z)\leq\gamma_0$ for all $z\in\RR$, and $\|f\|_1,\|f\|_2\leq\gamma_0$.
\item Suppose that $n^{-\frac12}\bink{x}y\in\III$. Let
	\begin{eqnarray}\label{eqlambda}
	\lambda&=&\frac{d\sigma(\rho^d-\rho)}{\sigma_\order(1-c(d-1)\rho^{d-1})}\mbox{ and}\\
	F(z)&=&\frac{n}{2\pi\rho^{d/2}\sigma\sigma_\order}\exp\brk{-\frac12\bc{(x\sigma_\order^{-1}-z\lambda)^2
			+\rho^d(y\rho^{-d}\sigma^{-1}-c\rho^{-1}\sigma^{-1}x+z)^2}}.	\nonumber
	\end{eqnarray}
	Then $\abs{f(z)-F(z)}\leq\gamma^{-2}$ for all $z\in\brk{-\gamma,\gamma}$.
	If $|z|>\gamma_0$, then $|f(z)|\leq\exp(-z^2/\gamma_0)+O(n^{-90})$.
\end{enumerate}
\end{lemma}
We defer the proof of \Lem~\ref{Lemma_fexplicit} to \Sec~\ref{Sec_fexplicit}.
Note that \Lem~\ref{Lemma_fexplicit} provides an explicit expression $F(z)$ that approximates $f(z)$ well on compact sets,
and shows that $f(z)\rightarrow0$ rapidly as $z\rightarrow\infty$.
Indeed, $F(z)$ just reflects~(\ref{eqBivOut1}).

Furthermore, the following lemma, whose proof we defer to \Sec~\ref{Sec_gcont}, shows that $g$ enjoys a certain
``continuity'' property.

\begin{lemma}\label{Lemma_gcont}
For any $\alpha>0$ there are $\beta>0$ and $n_0>0$ so that
for all $n\geq n_0$ and $z,z'\in\brk{-z^*,z^*}$ such that $|z-z'|<\beta$ we have
$g(z')\leq(1+\alpha)g(z)+n^{-20}$.
\end{lemma}
Further, in \Sec~\ref{Sec_fg} we shall combine \Lem s~\ref{Lemma_fexplicit} and~\ref{Lemma_gcont} to restate~(\ref{eqBivOut2}) as follows.
\begin{lemma}\label{Lemma_fg}
We have $f(z)=(1+o(1))(g*\phi(z))+O(n^{-18})$ for all $z\in\RR$.
\end{lemma}
Since $f$ is bounded and both $f$ and $g$ vanish outside of the interval $\brk{-z^*,z^*}$, \Lem~\ref{Lemma_fg}
entails that $\|f-g*\phi\|_2=o(1)$.
In addition, 
we infer the following bound on $g$.

\begin{corollary}\label{Lemma_FourierAux1}
There is a number $0<K=O(1)$ such that $g(z)\leq Kf(z)+O(n^{-18})$ for all $z\in\brk{-z^*,z^*}$.
Hence, $\|g\|_1,\|g\|_2=O(1)$.
\end{corollary}
\begin{proof}
Let $z\in\brk{-z^*,z^*}$.
By \Lem~\ref{Lemma_gcont} there is a number $\beta>0$ such that
$g(z')\geq\frac12g(z)-n^{-20}$ for all $z'\in\brk{-z^*,z^*}$ that satisfy $|z-z'|\leq\beta$.
Therefore, \Lem~\ref{Lemma_fg} entails that
	\begin{eqnarray*}
	f(z)&=&(1+o(1))\int g(z+\zeta)\phi(\zeta)d\zeta+O(n^{-18})\\
		&\geq&\frac{g(z)}{2+o(1)}\int_{\brk{-z^*,z^*}\cap\brk{z-\beta,z+\beta}}\phi(\zeta)d\zeta+O(n^{-18})
		\geq\frac{\beta g(z)}{10}+O(n^{-18}),
	\end{eqnarray*}
whence the desired estimate follows.
\qed\end{proof}

To obtain an explicit formula for $g$, we exhibit another function $h$ such that $\|f-h*\phi\|_2=o(1)$.

\begin{lemma}\label{Lemma_fh}
Suppose that $n^{-\frac12}\bink{x}y\in\III$, let $\lambda$ be as in~(\ref{eqlambda}), and define
\begin{eqnarray}\nonumber
\chi &=& 	\lambda^2 + \rho^d,\ 
\kappa = -\brk{\frac{\lambda}{\sigma_\cN} + \frac{c\rho^{d-1}}{\sigma}}x + \frac{y}{\sigma},\nonumber\ 
\theta = \frac{x^2}{\sigma_\cN^2} + \frac{(c\rho^{d-1} x - y)^2}{\rho^d\sigma^2},\mbox{ and}\nonumber\\
h(z)&=& \frac{n}{2\pi\rho^{d/2}\sqrt{1-\chi}\sigma_\cN\sigma}\exp\brk{
-\frac{\chi\theta-\kappa^2}{2\chi}-\frac{\left(\chi z+\kappa\right)^2}{2(\chi-\chi^2)}},
\label{eq:defhnmz}
\end{eqnarray}
Then $\|f-h*\phi\|_2=o(1)$.
\end{lemma}

The proof of \Lem~\ref{Lemma_fh} can be found in \Sec~\ref{Sec_fh}.
Thus, we have the two relations $\|f-g*\phi\|_2=o(1)$ and $\|f-h*\phi\|_2=o(1)$.
In \Sec~\ref{Sec_Fourier} we shall see that these bounds imply that actually $h$ approximates $g$ pointwise.

\begin{lemma}\label{Prop_Fourier}
For any $\alpha>0$ there is $n_0>0$ such that for all $n>n_0$, all $z\in\brk{-z^*/2,z^*/2}$,
and all $\nu,\bar\mu$ such that $n^{-\frac12}\bink{x}y\in\III$ we have
	$|g(z)-h(z)|<\alpha$.
\end{lemma}

In summary, by now we have obtained an explicit formula for $g(z)$ by rephrasing~(\ref{eqBivOut2}) in terms of $f$ and $g$ as $\|f-g*\phi\|_2=o(1)$.
Since \Thm~\ref{Thm_Nlocal} yields an explicit formula for $f$, we have been able to compute $g$. 
In particular, we have an asymptotic formula for $g(0)=\pr\brk{\order(H_d(n,m_0))=\nu\wedge\bar\size(H_d(n,m_0))=\bar\mu}$;
let us point out that this implies \Thm~\ref{Thm_Hnmlocal}.

\smallskip\noindent\emph{Proof of \Thm~\ref{Thm_Hnmlocal}.}
Suppose that $n^{-\frac12}\bink{x}y\in\III$.
Let $\mu=m_0-\bar\mu$.
Since $\size(H_d(n,m_0))=m_0-\bar\size(H_d(n,m_0))$, we have
	$g(0)=\pr\brk{\order(H_d(n,m_0))=\nu\wedge\size(H_d(n,m_0))=\mu}.$
Furthermore, $|h(0)-g(0)|<\alpha$ by Lemma~\ref{Prop_Fourier}.
Moreover, it is elementary though tedious to verify that
$h(0)=Q(\nu-(1-\rho)n,\mu-(1-\rho^d)m_0)$, where $Q$ is the function defined in \Thm~\ref{Thm_Hnmlocal}.
\qed

\medskip
Finally, to derive \Thm~\ref{Thm_Hnplocal} from \Thm~\ref{Thm_Hnmlocal}, we employ the relation
	\begin{eqnarray}\nonumber
	\pr\brk{\order(\hnp)=\nu\wedge\size(\hnp)=\mu}\\
		&\hspace{-8cm}=&\hspace{-4cm}\;
			\sum_{m=0}^{\bink{n}d}
				\pr\brk{\order(\hnm)=\nu\wedge\size(\hnm)=\mu}\cdot\pr\brk{\Bin\bc{\bink{n}d},p=m},
	\label{eqRueckschritt0}
	\end{eqnarray}
whose r.h.s.\ we know due to \Thm~\ref{Thm_Hnmlocal}.
We defer the details to \Sec~\ref{Sec_Hnplocal}.

\subsection{Proof of \Lem~\ref{Prop_Fourier}}\label{Sec_Fourier}

We normalize the Fourier transform as
	$\hat \varphi(\xi)=(2\pi)^{-\frac12}\int_{\RR}\varphi(\zeta)\exp(i\zeta\xi)d\zeta$, so that
the Plancherel theorem yields
	\begin{equation}\label{eqPlancherel}
	\| \varphi\|_2=\|\hat  \varphi\|_2,\qquad\mbox{provided that $ \varphi\in L_1(\RR)\cap L_2(\RR)$}.
	\end{equation}

Note that the proof of \Lem~\ref{Prop_Fourier} would be easy if it were true that $f=g*\phi$ and $f=h*\phi$.
For in this case we could just Fourier transform $f$ to obtain $\hat f=\hat g\hat\phi=\hat h\hat\phi$.
Then, dividing by $\hat\phi=\phi$ would yield $\hat g=\hat h$, and Fourier transforming once more we would get $g=h$.
However, since we do not have $f=g*\phi$ and $f=h*\phi$, but only $\|f-g*\phi\|_2,\|f-h*\phi\|_2=o(1)$, we have to work a little.

\Lem s~\ref{Lemma_fg} and~\ref{Lemma_fh} imply that there is a function $\omega=\omega(n)$ such that
$\lim_{n\rightarrow\infty}\omega(n)=\infty$ and
	$\|f-g*\phi\|_2,\|f-h*\phi\|_2<\frac12\exp(-\omega^2)$.
Thus,
	\begin{equation}\label{eqsmalltauI}
	\|(g-h)*\phi\|_2<\exp(-\omega^2)=o(1).
	\end{equation}
In order to compare $g$ and $h$, the crucial step is to establish that actually
$\|(g-h)*\phi_{0,\tau^2}\|_2=o(1)$ for ``small'' numbers $\tau<1$;
indeed, we are mainly interested in $\tau=o(1)$.
We point out that by \Lem~\ref{Lemma_fexplicit} and
\Cor~\ref{Lemma_FourierAux1} we can apply the Plancherel theorem~(\ref{eqPlancherel}) to both $f$ and $g$,
because $f,g\in L_1(\RR)\cap L_2(\RR)$.

\begin{lemma}\label{Lemma_smalltau}
Suppose that $\omega^{-1/8}\leq\tau\leq1$.
Then $\|(g-h)*\phi_{0,\tau^2}\|_2\leq\exp(-\omega/5)$.
\end{lemma}
\begin{proof}
Let $\xi=\hat\phi_{0,\tau^2}=\phi_{0,\tau^{-2}}$.
Then
	\begin{eqnarray}
	\|(g-h)*\phi_{0,\tau^2}\|_2^2
		&\stacksign{(\ref{eqPlancherel})}{=}&\|(\hat g-\hat h)\xi\|_2^2
		=\int_{-\omega}^{\omega}|(\hat g-\hat h)\xi|^2+\int_{\RR\setminus\brk{-\omega,\omega}}|(\hat g-\hat h)\xi|^2.
	\label{eqsmalltau1}
	\end{eqnarray}
Since $\hat\phi=\phi$, we obtain
	\begin{eqnarray}\nonumber
	\int_{-\omega}^{\omega}|(\hat g-\hat h)\xi|^2&\leq&
		\frac{\|\xi\|_{\infty}}{\inf_{-\omega\leq t\leq\omega}|\hat\phi(t)|^2}\int_{-\omega}^{\omega}|(\hat g-\hat h)\hat\phi|^2\\
			&\leq&\exp(\omega^2)\|(\hat g-\hat h)\hat\phi\|_2^2
			\;\stacksign{(\ref{eqPlancherel})}{=}\;\exp(\omega^2)\|(g-h)*\phi\|_2^2
			\;\stacksign{(\ref{eqsmalltauI})}{\leq}\;\exp(-\omega^2).
	\label{eqsmalltau2}
	\end{eqnarray}
In addition, by the Cauchy-Schwarz inequality
	\begin{eqnarray}
	\int_{\RR\setminus\brk{-\omega,\omega}}|(\hat g-\hat h)\xi|^2
		&\leq&\brk{\int_{\RR}|(\hat g-\hat h)^2|^2}^{\frac12}\cdot\brk{\int_{\RR\setminus\brk{-\omega,\omega}}|\xi|^4}^{\frac12}
	\label{eqsmalltau3}
	\end{eqnarray}
Furthermore, as $\tau^{-2}\leq\omega^\frac14$, we have
	\begin{equation}\label{eqsmalltau4}
	\int_{\RR\setminus\brk{-\omega,\omega}}|\xi|^4\leq\tau^{-2}\int_{\omega}^{\infty}\exp(-2\tau^2\zeta^2)d\zeta\leq\exp(-\omega).
	\end{equation}
Moreover, by \Cor~\ref{Lemma_FourierAux1}
	\begin{eqnarray}\nonumber
	\int_{\RR}|(\hat g-\hat h)^2|^2&=&\|(\hat g-\hat h)^2\|_2^2\;\stacksign{(\ref{eqPlancherel})}{=}\;\|(g-h)*(g-h)\|_2^2
		\leq\brk{\|g*g\|_2+2\|g*h\|_2+\|h*h\|_2}^2\\
		&\leq&\brk{K^2\|f*f\|_2+2K\|f*h\|_2+\|h*h\|_2}^2+o(1).
	\label{eqsmalltau5}
	\end{eqnarray}
Considering the bounds on $f$ and $h$ obtained in \Lem s~\ref{Lemma_fexplicit}
and~\ref{Lemma_fh}, we see that $\|f*f\|_2,\|f*h\|_2,\|h*h\|_2=O(1)$.
Therefore, (\ref{eqsmalltau3}), (\ref{eqsmalltau4}), and~(\ref{eqsmalltau5}) imply that
	\begin{equation}\label{eqsmalltau6}
	\int_{\RR\setminus\brk{-\omega,\omega}}|(\hat g-\hat h)\xi|^2\leq O(\exp(-\omega/2)).
	\end{equation}
Finally, combining~(\ref{eqsmalltau1}), (\ref{eqsmalltau2}), and~(\ref{eqsmalltau6}), 
we obtain the desired bound on $\|(g-h)*\phi_{0,\tau^2}\|_2$.
\qed\end{proof}

In order to complete the proof of \Lem~\ref{Prop_Fourier}, we show that \Lem~\ref{Lemma_smalltau} implies that
actually $g(z)$ must be close to $h(z)$ for all points $z\in\brk{-z^*/2,z^*/2}$.
The basic idea is as follows.
For ``small'' $\tau$ the function $\phi_{0,\tau^2}$ is a narrow ``peak'' above the origin.
Therefore, the continuity property of $g$ established in \Lem~\ref{Lemma_gcont} implies that the convolution $g*\phi_{0,\tau^2}(z)$
is ``close'' to the function $g(z)$ itself.
Similarly, $h*\phi_{0,\tau^2}(z)$ is ``close'' to $h(z)$.
Hence, as $g*\phi_{0,\tau^2}(z)$ is ``close'' to $h*\phi_{0,\tau^2}(z)$ by \Lem~\ref{Lemma_smalltau}, we
can infer that $h(z)$ approximates $g(z)$.
Let us carry out the details.

\smallskip
\noindent\emph{Proof of \Lem~\ref{Prop_Fourier}.}
Assume for contradiction that there is some $z\in\brk{-z^*/2,z^*/2}$ and some fixed $0<\alpha=\Omega(1)$
such that $g(z)>h(z)+\alpha$ for arbitrarily large $n$
(an analogous argument applies in the case $g(z)<h(z)-\alpha$).
Let $\tau=\omega^{-1/8}$.
Our goal is to infer that
	\begin{equation}\label{eqPropFourier1}
	\|(h-g)*\phi_{0,\tau^2}\|_2>\exp(-\omega/5),
	\end{equation}
which contradicts \Lem~\ref{Lemma_smalltau}.

To show (\ref{eqPropFourier1}), note that \Cor~\ref{Lemma_FourierAux1} implies that $\|g\|_\infty=O(1)$,
because the bound $\|f\|_\infty=O(1)$ follows from \Lem~\ref{Lemma_fexplicit}.
Similarly, the function $h$ detailed in \Lem~\ref{Lemma_fh} is bounded.
Thus, let $\Gamma=O(1)$ be such that
		$g(\zeta),h(\zeta)\leq\Gamma\mbox{ for all }\zeta\in\RR.$
Then \Lem~\ref{Lemma_gcont} implies that there exists $0<\beta=\Omega(1)$ such that
	\begin{equation}\label{eqPropFourier3}
	(1-0.01\alpha\Gamma^{-1})g(z)-O(n^{-18})\leq g(z')\leq(1+0.01\alpha\Gamma^{-1})g(z)+O(n^{-18})
	\mbox{ if }|z-z'|<\beta.
	\end{equation}
In fact, as $h$ is continuous on $(-z^*,z^*)$, we can choose $\beta$ small enough so that in addition
	\begin{equation}\label{eqPropFourier4}
	(1-0.01\alpha\Gamma^{-1})h(z)-O(n^{-18})\leq h(z')\leq(1+0.01\alpha\Gamma^{-1})h(z)+O(n^{-18})
			\mbox{ if }|z-z'|<\beta.
	\end{equation}
Combining~(\ref{eqPropFourier3}) and~(\ref{eqPropFourier4}), we conclude that
	\begin{equation}\label{eqPropFourier4X}
	|g(z')-g(z'')|\leq0.1\alpha,\,|h(z')-h(z'')|\leq0.1\alpha
		\mbox{ for all $z',z''$ such that }|z-z'|,|z-z''|<\beta.
	\end{equation}

Further, let $\gamma=\int_{\RR\setminus\brk{-\beta/2,\beta/2}}\phi_{0,\tau^2}$.
Then for sufficiently large $n$ we have
$\gamma<0.01\alpha\Gamma^{-1}$, because $\tau\rightarrow0$ as $n\rightarrow\infty$.
Therefore, for any $z'$ such that $|z'-z|<\beta/2$ we have
	\begin{eqnarray}\nonumber
	g*\phi_{0,\tau^2}(z')&=&\int_\RR g(z'+\zeta)\phi_{0,\tau^2}(\zeta)d\zeta
		\geq\int_{-\beta/2}^{\beta/2} g(z'+\zeta)\phi_{0,\tau^2}(\zeta)d\zeta\\
		&\stacksign{(\ref{eqPropFourier4X})}{\geq}&(g(z)-0.01\alpha)(1-\gamma)\geq g(z)-0.02\alpha,
		\mbox{ and similarly}\label{eqPropFourier5}\\
	h*\phi_{0,\tau^2}(z')&\stacksign{(\ref{eqPropFourier4})}{\leq}&h(z)+0.02\alpha.\label{eqPropFourier6}
	\end{eqnarray}
Since~(\ref{eqPropFourier5}) and~(\ref{eqPropFourier6}) are true for all $z'$ such that $|z'-z|<\beta/2$,
our assumption $g(z)>h(z)+\alpha$ yields
	\begin{eqnarray}\label{eqPropFourier7}
	\|(g-h)*\phi_{0,\tau^2}\|_2^2&\geq&
		\int_{-\beta/2}^{\beta/2}|g*\phi_{0,\tau^2}(z')-h*\phi_{0,\tau^2}(z')|^2
			\geq0.5\alpha^2\beta.
	\end{eqnarray}
As $\alpha,\beta$ remain bounded away from $0$ while $\omega(n)\rightarrow\infty$ as $n\rightarrow\infty$, for sufficiently large $n$
we have $0.5\alpha^2\beta>\exp(-\omega/5)$, so that~(\ref{eqPropFourier7}) implies~(\ref{eqPropFourier1}).
\qed

\subsection{Proof of \Lem~\ref{Lemma_fexplicit}}\label{Sec_fexplicit}

To prove \Lem~\ref{Lemma_fexplicit}, we first establish~(\ref{eqBivOut1}) rigorously.
Then, we employ \Prop~\ref{Prop_Bin} and \Thm~\ref{Thm_Nlocal} to obtain explicit expressions for the r.h.s.\ of~(\ref{eqBivOut1}).

\begin{lemma}\label{Lemma_rhoz}
Let $z\in\brk{-z^*,z^*}$, 
$\mu_\order=(1-\rho)n$ and $\lambda=\frac{d\sigma(\rho^d-\rho)}{\sigma_\order(1-c(d-1)\rho^{d-1})}$.
\begin{enumerate}
\item Let $c_z=\bink{n-1}{d-1}p_z$.
	Then there is a unique $0<\rho_z<1$ such that $\rho_z=\exp(c_z(\rho_z^{d-1}-1))$.
	Moreover, $\Erw(\order(H_d(n,p_z)))=(1-\rho_z)n+o(\sqrt{n})=\mu_\order+z\sigma_\order\lambda+o(\sqrt{n})$.
\item Furthermore, $\pr\brk{\order(H_d(n-\nu,p_z))>\ln^2n},\pr\brk{\order(H_d(n-\nu,\bar\mu))>\ln^2n}\leq n^{-100}$.
\end{enumerate}
\end{lemma}
\begin{proof}
Since $c_z\sim c_0=\bink{n-1}{d-1}p>(d-1)^{-1}$,
\Thm~\ref{Thm_global} entails that for each $z\in\brk{-z^*,z^*}$ there exists a unique  $0<\rho_z<1$
such that $\rho_z=\exp(c_z(\rho_z^{d-1}-1))$.
Furthermore, 
the function $z\mapsto\rho_z$ is differentiable by the implicit function theorem.
Consequently, we can Taylor expand $\rho_z$ at $z=0$
by differentiating both sides of the transcendental equation $\rho_z=\exp(c_z(\rho_z^{d-1}-1))$, which yields
	\begin{equation}\label{eqrhoz}
	\rho_z=\rho+\lambda\sigma_\order n^{-1}z+o(n^{-1/2}).
	\end{equation}
Hence, as $\Erw(\order(H_d(n,p_z)))=(1-\rho_z)n+o(\sqrt{n})$
by  \Thm~\ref{Thm_global}, we obtain the first assertion.

The second part follows from \Thm~\ref{Thm_global} as well,
because by~(\ref{eqrhoz}) we have $\nu\sim(1-\rho_z)n\sim(1-\rho)n$ for all $z\in\brk{-z^*,z^*}$.
\qed\end{proof}

The basic reason why \Lem~\ref{Lemma_rhoz} implies~(\ref{eqBivOut1}) is the following.
Let $G\subset V$ be a set of size $\nu$.
If we condition on the event that $G$ is a component, then the hypergraph $H_d(n,p_z)-G$ obtained from $H_d(n,p_z)$ by removing the vertices in $G$
is distributed as $H_d(n-\nu,p_z)$.
For whether or not $G$ is a component does not affect the edges of $H_d(n,p_z)-G$.
Thus, \Lem~\ref{Lemma_rhoz} entails that $H_d(n,p_z)-G$ has no component of order $>\ln^2n$ \whp, whence $G$ is the largest component of $H_d(n,p_z)$.
Therefore, conditioning on the event that $G$ actually is the \emph{largest} component is basically equivalent to just conditioning
on the event that $G$ is a component, and in the latter case the number of edges
in $H_d(n,p_z)-G=H_d(n-\nu,p)$ is binomially distributed $\Bin(N,p_z)$, were we let $N=\bink{n-\nu}d$.
Let us now carry out this sketch in detail.

\begin{lemma}\label{Lemma_fexplicitAux}
We have $1-n^{-98}\leq\frac{f(z)}{n\pr\brk{\Bin(N,p)=\bar\mu}\pr\brk{\order(\hnp)=\nu}}\leq1+n^{-98}.$
\end{lemma}
\begin{proof}
Let $\mathcal{G}=\{G\subset V:|G|=\nu\}$.
For $G\in\mathcal{G}$ we let $\comp_G$ denote the event that $G$ is a component in $H_d(n,p_z)$.
Then by the union bound
	\begin{eqnarray}\label{eqfexplicit1}
	Q&\leq&\sum_{\GGG\in\mathcal{G}}\pr\brk{\comp_G\wedge|E(H_d(n,p_z)-\GGG)|=\bar\mu}
		=\sum_{\GGG\in\mathcal{G}}\pr\brk{\comp_G}\pr\brk{|E(H_d(n,p_z)-\GGG)|=\bar\mu}.
	\end{eqnarray}
As $H_d(n,p_z)-G$ is the same as $H_d(n-\nu,p_z)$,
$|E(H_d(n,p_z)-\GGG)|$ is binomially distributed with parameters $N$ and $p_z$.
Moreover, 
$\pr\brk{\comp_G\wedge\order(H_d(n,p_z)-\GGG)<\nu}=\pr\brk{\comp_G}\pr\brk{\order(H_d(n,p_z)-\GGG)<\nu}$.
Therefore, (\ref{eqfexplicit1}) yields
	\begin{eqnarray}\nonumber
	{f(z)}n^{-1}&\leq&\pr\brk{\Bin(N,p_z)=\bar\mu}\sum_{\GGG\in\mathcal{G}}\pr\brk{\comp_G}\\
		&=&\pr\brk{\Bin(N,p_z)=\bar\mu}\sum_{\GGG\in\mathcal{G}}
			\frac{\pr\brk{\comp_G\wedge\order(H_d(n,p_z)-\GGG)<\nu}}{\pr\brk{\order(H_d(n,p_z)-\GGG)<\nu}}.
		\label{eqfexplicit2}
	\end{eqnarray}
Furthermore, $\pr\brk{\order(H_d(n,p_z)-\GGG)<\nu}\geq1-n^{-100}$ by the 2nd part of \Lem~\ref{Lemma_rhoz}.
Thus, (\ref{eqfexplicit2}) entails
	\begin{eqnarray}\nonumber
	(1-n^{-100})\pr\brk{\Bin(N,p_z)=\bar\mu}^{-1}n^{-1}f(z)&\leq&
		\sum_{\GGG\in\mathcal{G}}\pr\brk{\comp_G\wedge\order(H_d(n,p_z)-\GGG)<\nu}\\
		&\hspace{-6cm}=&\hspace{-3cm}\;
			\pr\brk{\exists G\in\mathcal{G}:\comp_G\wedge\order(H_d(n,p_z)-\GGG)<\nu}
		\leq\pr\brk{\order(H_d(n,p_z))=\nu}.
	\label{eqfexplicit3}
	\end{eqnarray}

Conversely, if $G\in\mathcal{G}$ is a component of $H_d(n,p_z)$ and $\order(H_d(n,p_z)-G)<\nu$, then
$G$ is the unique largest component of $H_d(n,p_z)$.
Therefore,
	\begin{eqnarray*}\nonumber
	n^{-1}f(z)&\geq&\sum_{\GGG\in\mathcal{G}}\pr\brk{\comp_G\wedge\order(H_d(n,p_z)-\GGG)<\nu\wedge|E(H_d(n,p_z)-\GGG)|=\bar\mu}\\
	&=&\sum_{\GGG\in\mathcal{G}}\pr\brk{\comp_G}\pr\brk{\order(H_d(n,p_z)-\GGG)<\nu\wedge|E(H_d(n,p_z)-\GGG)|=\bar\mu}.
		\label{eqfexplicit4}
	\end{eqnarray*}
Further, given that $|E(H_d(n,p_z)-\GGG)|=\bar\mu$, $H_d(n,p_z)-G$ is just a random hypergraph $H_d(n-\nu,\bar\mu)$.
Hence, (\ref{eqfexplicit4}) yields
	\begin{eqnarray}\nonumber
	n^{-1}f(z)&\geq&\pr\brk{\order(H_d(n-\nu,\bar\mu))<\nu}\pr\brk{\Bin(N,p_z)=\bar\mu}\sum_{\GGG\in\mathcal{G}}\pr\brk{\comp_G}\\
	&\geq&\pr\brk{\order(H_d(n-\nu,\bar\mu))<\nu}\pr\brk{\Bin(N,p_z)=\bar\mu}\pr\brk{\order(H_d(n,p_z))=\nu},
		\label{eqfexplicit5}
	\end{eqnarray}
where the last estimate follows from the union bound.
Moreover, 
by the 2nd part of \Lem~\ref{Lemma_rhoz}.
Plugging this into~(\ref{eqfexplicit5}), we get
	\begin{equation}\label{eqfexplicit8}
	n^{-1}f(z)\geq(1-n^{-99})\pr\brk{\Bin(N,p_z)=\bar\mu}\pr\brk{\order(H_d(n,p_z))=\nu}.
	\end{equation}
Combining~(\ref{eqfexplicit3}) and~(\ref{eqfexplicit8}) completes the proof.
\qed\end{proof}

\noindent\emph{Proof of \Lem~\ref{Lemma_fexplicit}.}
Suppose that $|x|,|y|\leq\sqrt{n}\ln n$.
Then \Thm~\ref{Thm_Nlocal} entails that $\pr\brk{\order(H_d(n,p_z)=\nu}=O(n^{-\frac12})$,
and \Prop~\ref{Prop_Bin} yields $\pr\brk{\Bin(N,p_z)=\bar\mu}=O(n^{-\frac12})$.
Thus, the assertion follows from \Lem~\ref{Lemma_fexplicitAux}.

With respect to the 2nd assertion, suppose that $n^{-\frac12}\bink{x}y\in\III$,
fix some $\gamma>0$, and consider $z\in\brk{-\gamma,\gamma}$.
Let $c_z=\bink{n-1}{d-1}p_z$, and let $0<\rho_z<1$ be the unique solution to to $\rho_z=\exp(c_z(\rho_z^{d-1}-1))$ (cf.\ \Lem~\ref{Lemma_rhoz}).
In addition, let
	$$\mu_{\order,z}=(1-\rho_z)n,\ \sigma_{\order,z}=\frac{\sqrt{\rho_z\left(1-\rho_z + c_z(d-1)(\rho_z-\rho_z^{d-1})\right)n}}{1-c_z(d-1)\rho_z^{d-1}},$$
and set $\sigma_{\order}=\sigma_{\order,0}$.
Then \Thm~\ref{Thm_Nlocal} implies that
	\begin{eqnarray}\nonumber
	\pr\brk{\order(H_d(n,p_z))=\nu}&\sim&\frac1{\sqrt{2\pi}\sigma_{\order,z}}\exp\bc{-\frac{(\nu-\mu_{\order,z})^2}{2\sigma_{\order,z}^2}}\\
		&\stacksign{\Lem~\ref{Lemma_rhoz}}{\sim}&
			\frac1{\sqrt{2\pi}\sigma_{\order}}\exp\bc{-\frac{(\nu-(1-\rho)n-z\lambda\sigma_\order)^2}{2\sigma_{\order}^2}}.
		\label{eqfexplicitNlocal}
	\end{eqnarray}
In addition, since 
$Np_z=\bink{n-\nu}{d}(p+z\sigma\bink{n}{d}^{-1})=\rho^d(m_0+z\sigma-c\rho^{-1}x)+o(\sqrt{n})$, 
\Prop~\ref{Prop_Bin} entails that
	\begin{eqnarray}
	\pr\brk{\Bin(N,p_z)=\bar\mu}
		&\sim&\frac1{\sqrt{2\pi\rho^dm_0}}\exp\bc{-\frac{(\bar\mu-\rho^d(m_0+z\sigma-c\rho^{-1}x))^2}{2\rho^dm_0}}.
		\label{eqfexplicitNlocalB}
	\end{eqnarray}	
Hence, \Lem~\ref{Lemma_fexplicitAux} yields
	\begin{eqnarray*}
	n^{-1}f(z)&\sim&\pr\brk{\order(H_d(n,p_z))=\nu}\pr\brk{\Bin(N,p_z)=\bar\mu}\\
		&\stacksign{(\ref{eqfexplicitNlocal}), (\ref{eqfexplicitNlocalB})}{\sim}&\frac1{2\pi\sqrt{\rho^dm_0}\sigma_{\order}}\exp\bc{-\frac{(\nu-(1-\rho)n-z\lambda\sigma_\order)^2}{2\sigma_{\order}^2}
			-\frac{(\bar\mu-\rho^d(m_0+z\sigma-c\rho^{-1}x))^2}{2\rho^dm_0}}\\
		&\sim&\frac1{2\pi\rho^{d/2}\sigma\sigma_{\order}}\exp\bc{-\frac{(x-z\lambda\sigma_{\order})^2}{2\sigma_{\order}^2}
			-\frac{(y+\rho^d\sigma z-c\rho^{d-1}x)^2}{2\rho^d\sigma^2}}
=n^{-1}F(z),
	\end{eqnarray*}
so that we have established the first assertion.
	
Finally, let us assume that $\gamma_0<|z|\leq|z^*|$ for some large enough but fixed $\gamma_0>0$.
Then $|Np_z-\bar\mu|=\Omega(z\sqrt{n})$.
Therefore, \Prop~\ref{Prop_Bin} implies that $\pr\brk{\Bin(n,p_z)=\bar\mu}\leq n^{-1/2}\exp(-\Omega(z^2))$.
Furthermore, $\pr\brk{\order(H_d(n,p_z))=\nu}=O(n^{-1/2})$ by \Thm~\ref{Thm_Nlocal}.
Hence, \Lem~\ref{Lemma_fexplicitAux} entails that $f(z)\leq O(\exp(-\Omega(z^2))+n^{-97})$,
as desired.
\qed

\subsection{Proof of \Lem~\ref{Lemma_gcont}}\label{Sec_gcont}

Throughout this section we assume that $z,z'\in\brk{-z^*,z^*}$, and that $|z-z'|<\beta$ for some small $\beta>0$.
In addition, we may assume that
	\begin{equation}\label{eqgcont0}
	g(z')\geq n^{-30},
	\end{equation}
because otherwise the assertion is trivially true.
To compare $g(z)$ and $g(z')$, we first express $g(z)$ in terms of the number $C_d(\nu,m_z-\bar\mu)$ of  connected
$d$-uniform hypergraphs of order $\nu$ and size $m_z-\bar\mu$.

\begin{lemma}\label{Lemma_gexpression}
We have $\bink{\bink{n}d}{m_z}g(z)\sim n\bink{n}{\nu}C_d(\nu,m_z-\bar\mu)\bink{\bink{n-\nu}d}{\bar\mu}$.
A similar statement is true for $g(z')$.
\end{lemma}
\begin{proof}
We claim that
	\begin{equation}\label{eqgcont1}
	n^{-1}g(z)\leq\bink{n}{\nu}C_d(\nu,m_z-\bar\mu)\bink{\bink{n-\nu}d}{\bar\mu}\bink{\bink{n}d}{m_z}^{-1}.
	\end{equation}
The reason is that $n^{-1}g(z)$ is the probability that the largest component of $H_d(n,m_z)$
has order $\nu$ and size $m_z-\bar\mu$, while
the right hand side equals the \emph{expected} number of such components.
For there are $\bink{n}{\nu}$ ways to choose $\nu$ vertices where to place such a component.
Then, there are $C_d(\nu,m_z-\bar\mu)$ ways to choose the component itself.
Moreover, there are $\bink{\bink{n-\nu}d}{\bar\mu}$ ways to choose the hypergraph induced on the remaining $n-\nu$ vertices,
while the total number of $d$-uniform hypergraphs of order $n$ and size $m_z$ is $\bink{\bink{n}d}{m_z}$.
Conversely,
	\begin{equation}\label{eqgcont2}
	n^{-1}g(z)\geq\bink{n}{\nu}C_d(\nu,m_z-\bar\mu)\bink{\bink{n-\nu}d}{\bar\mu}\pr\brk{\order(H_d(n-\nu,\bar\mu))<\nu}\bink{\bink{n}d}{m_z}^{-1}.
	\end{equation}
For the right hand side equals the probability that $H_d(n,m_z)$ has one component of order $\nu$ and size $m_z-\bar\mu$,
while all other components have order $<\nu$.
Since $\pr\brk{\order(H_d(n-\nu,\bar\mu))<\nu}\sim1$ by \Lem~\ref{Lemma_rhoz}, the assertion follows from
(\ref{eqgcont1}) and~(\ref{eqgcont2}).
\qed\end{proof}

\Lem~\ref{Lemma_gexpression} entails that
	\begin{equation}\label{eqgcont3}
	\frac{g(z')}{g(z)}\sim\frac{C_d(\nu,m_{z'}-\bar\mu)}{C_d(\nu,m_z-\bar\mu)}\cdot\frac{\bink{\bink{n}d}{m_{z}}}{\bink{\bink{n}d}{m_{z'}}}.
	\end{equation}
Thus, as a next step we estimate the two factors on the r.h.s.\ of~(\ref{eqgcont3}).

\begin{lemma}\label{Lemma_Cquotient}
If $|z-z'|<\beta$ for a small enough $\beta>0$, then
$\frac{C_d(\nu,m_{z'}-\bar\mu)}{C_d(\nu,m_z-\bar\mu)}\cdot p^{m_z-m_{z'}}\leq1+\alpha/2$.
\end{lemma}
To prove \Lem~\ref{Lemma_Cquotient}, we employ the following estimate, which we will
establish in \Sec~\ref{Sec_Cquotient}.
\begin{lemma}\label{Lemma_CquotientAux}
If $|z-z'|<\beta$ for a small enough $\beta>0$, then letting
	\begin{eqnarray*}
	P&=&\pr\brk{\order(H_d(n,p_{z'}))=\nu\wedge\size(H_d(n,p_{z'}))=m_{z}-\bar\mu},\\
	P'&=&\pr\brk{\order(H_d(n,p_{z'}))=\nu\wedge\size(H_d(n,p_{z'}))=m_{z'}-\bar\mu},
	\end{eqnarray*}
we have $(1-\alpha/3)P-n^{-80}\leq P'\leq(1+\alpha/3)P+n^{-80}$.
\end{lemma}
\emph{Proof of \Lem~\ref{Lemma_Cquotient}.}
We observe that
	\begin{eqnarray}\label{eqCquotient0}
	P&\leq&\bink{n}{\nu}C_d(\nu,m_z-\bar\mu)p_{z'}^{m_{z}-\bar\mu}(1-p_{z'})^{\bink{n}{d}-\bink{n-\nu}{d}-(m_z-\bar\mu)},
	\end{eqnarray}
because the r.h.s.\ equals the \emph{expected} number of components of order $\nu$ and size $m_z-\bar\mu$ in $H_d(n,p_{z'})$.
(For there are $\bink{n}{\nu}$ ways to choose the $\nu$ vertices where to place the component and $C_d(\nu,m_z-\bar\mu)$ ways to
choose the component itself.
Furthermore, edges are present with probability $p_{z'}$ independently, and thus
the $p_{z'}^{m_z-\bar\mu}$ factor accounts for the presence of the $m_z-\bar\mu$ desired edges among the selected $\nu$ vertices.
Moreover, the $(1-p_{z'})$-factor rules out further edges among the $\nu$ chosen vertices and in-between the $\nu$ chosen and the
$n-\nu$ remaining vertices.)
Conversely,
	\begin{eqnarray}\label{eqCquotient1}
	P&\geq&\bink{n}{\nu}C_d(\nu,m_{z}-\bar\mu)p_{z'}^{m_{z}-\bar\mu}(1-p_{z'})^{\bink{n}{d}-\bink{n-\nu}{d}-(m_z-\bar\mu)}
		\pr\brk{\order(H_d(n-\nu,p_{z'})<\nu)};
	\end{eqnarray}
for the r.h.s.\ is the probability that
there occurs exactly one component of order $\nu$ and size $m_z-\bar\mu$, while all other components have order $<\nu$.
As \Lem~\ref{Lemma_rhoz} entails that
	$\pr\brk{\order(H_d(n-\nu,p_{z'})<\nu)}\sim1$,
(\ref{eqCquotient0}) and~(\ref{eqCquotient1}) yield
	\begin{eqnarray*}\label{eqCquotient2}
	P&\sim&\bink{n}{\nu}C_d(\nu,m_{z}-\bar\mu)p_{z'}^{m_{z}-\bar\mu}(1-p_{z'})^{\bink{n}{d}-\bink{n-\nu}{d}-(m_z-\bar\mu)}\mbox{, and similarly}\\
	P'&\sim&\bink{n}{\nu}C_d(\nu,m_{z'}-\bar\mu)p_{z'}^{m_{z'}-\bar\mu}(1-p_{z'})^{\bink{n}{d}-\bink{n-\nu}{d}-(m_{z'}-\bar\mu)}.\label{eqCquotient3}
	\end{eqnarray*}
Therefore,
	\begin{eqnarray}
	\frac{C_d(\nu,m_{z'}-\bar\mu)}{C_d(\nu,m_{z}-\bar\mu)}&\sim&
		\frac{P'}{P}\cdot p_{z'}^{m_{z'}-m_z}\cdot (1-p_{z'})^{m_z-m_{z'}}
		\sim \frac{P'}{P}\cdot p^{m_{z'}-m_z}\nonumber\\
		&\stacksign{\Lem~\ref{Lemma_CquotientAux}}{\leq}&\;\bc{1+\frac{\alpha}3+\frac2{n^{80}P'-2}}p^{m_{z'}-m_z}.
		\label{eqCquotient4}
	\end{eqnarray}
In order to show that the r.h.s.\ of~(\ref{eqCquotient4}) is $\leq1+\alpha/2$, we need to lower bound $P'$: by \Prop~\ref{Prop_Bin}
	\begin{eqnarray}
	P'		&\geq&\pr\brk{\order(H_d(n,m_{z'}))=\nu\wedge\size(H_d(n,m_{z'}))=m_{z'}-\bar\mu}
			\cdot\pr\brk{\Bin\bc{\bink{n}d,p_{z'}}=m_{z'}}\nonumber\\
		&\geq&n^{-1}g(z') \;\stacksign{(\ref{eqgcont0})}{\geq}\; n^{-31}.
	\label{eqCquotient55}
	\end{eqnarray}
Finally, combining~(\ref{eqCquotient4}) and~(\ref{eqCquotient55}), we obtain the desired
bound on $C(\nu,m_{z'}-\bar\mu)$.
\qed

\begin{lemma}\label{Lemma_Binomquotient}
We have $\bink{\bink{n}d}{m_{z'}}\bink{\bink{n}d}{m_{z}}^{-1}=\exp(O(z-z')^2)\cdot p^{m_z-m_{z'}}$.
\end{lemma}
\begin{proof}
By Stirling's formula,
	\begin{eqnarray}
	\bink{\bink{n}d}{m_{z'}}\bink{\bink{n}d}{m_z}^{-1}&\sim&
		\bcfr{\bink{n}d}{m_{z'}}^{m_{z'}}\bcfr{\bink{n}d}{\bink{n}d-m_{z'}}^{\bink{n}d-m_{z'}}
			\brk{\bcfr{\bink{n}d}{m_{z}}^{m_{z}}\bcfr{\bink{n}d}{\bink{n}d-m_{z}}^{\bink{n}d-m_{z}}}^{-1}\nonumber\\
	&\sim&\frac{p_z^{m_z}}{p_{z'}^{m_{z'}}}\bc{1+\frac{m_{z'}}{\bink{n}d-m_{z'}}}^{\bink{n}d-m_{z'}}\bc{1+\frac{m_z}{\bink{n}d-m_z}}^{m_z-\bink{n}d}\nonumber\\
	&\sim&\frac{p_z^{m_z}}{p_{z'}^{m_{z'}}}\exp(m_{z'}-m_z)
		\sim p^{m_z-m_{z'}}\bcfr{p_z}{p_{z'}}^{m_{z'}}\exp(\sigma(z'-z)),\quad\mbox{where}
		\label{eqBinomquotient1}\\
	\bcfr{p_z}{p_{z'}}^{m_{z'}}&\sim&\bcfr{m_0+z\sigma}{m_0+z'\sigma}^{m_{z'}}
		\sim\exp\bc{(z-z')\sigma_0-\frac{(z-z')^2\sigma^2}{2m_{z'}}}\nonumber\\
		&=&\exp\bc{(z-z')\sigma_0-O(z-z')^2}
	\label{eqBinomquotient2}
	\end{eqnarray}
Combining~(\ref{eqBinomquotient1}) and~(\ref{eqBinomquotient2}), we obtain the assertion.
\qed\end{proof}
Plugging the estimates from \Lem s~\ref{Lemma_Cquotient} and~\ref{Lemma_Binomquotient} into~(\ref{eqgcont3}), we conclude that
$1-\alpha\leq g(z)/g(z')\leq1+\alpha$, provided that $|z-z'|<\beta$ for some small enough $\beta>0$,
thereby completing the proof of \Lem~\ref{Lemma_gcont}.

\subsection{Proof of \Lem~\ref{Lemma_Cquotient}}\label{Sec_Cquotient}

By symmetry, it suffices to prove that $P'\leq(1+\alpha/3)P+n^{-90}$.
To show this, we expose the edges of $H_d(n,p_{z'})$ in three rounds.
Let $\eps>0$ be a small enough number that remains fixed as $n\rightarrow\infty$.
Moreover, set $q_1=(1-\eps)p_{z'}$, and let $q_2\sim\eps p_{z'}$ be such that $q_1+q_2-q_1q_2=p_{z'}$.
Choosing $\eps>0$ sufficiently small, we can ensure that $\bink{n-1}{d-1}q_1>(d-1)^{-1}+\eps$.
Now, we construct $H_d(n,p_{z'})$ in three rounds as follows.
\begin{description}
\item[1st round.]
	Construct a random hypergraph $H_1$ with vertex set $V=\{1,\ldots,n\}$ by including each of the $\bink{n}d$ possible edges
	with probability $q_1$ independently.
	Let $G_1$ be the largest component of $H_1$.
\item[2nd round.]
	Let $H_2$ be the hypergraph obtained by adding  with probability $q_2$ independently
	each possible edge $e\not\in H_1$ that is not entirely contained in $G_1$ (i.e., $e\not\subset G_1$) to $H_1$.
	Let $G_2$ signify the largest component of $H_2$.
\item[3rd round.]
	Finally, obtain $H_3$ by adding each edge $e\not\in H_1$ such that $e\subset G_1$ with probability $q_2$ independently.
	Let $\mathcal{F}$ denote the set of edges added in this way.
\end{description}
Since for each of the $\bink{n}d$ possible edges the overall probability of being contained in $H_3$ is $q_1+(1-q_1)q_2=p_{z'}$,
$H_3$ is just a random hypergraph $H_d(n,p_{z'})$.
Moreover, as in the 3rd round we only add edges that fall completely into the component of $H_2$ that contains $G_1$,
we have $\order(H_d(n,p_{z'}))=\order(H_3)=\order(H_2)$.
Furthermore, $|\mathcal{F}|$ has a binomial distribution
	\begin{equation}\label{eqCquotient1X}
	|\mathcal{F}|=\Bin\bc{\bink{|G_1|}d-\size(H_1),q_2}.
	\end{equation}

To compare $P'$ and $P$, we make use of the local limit theorem for the binomially distributed $|\mathcal{F}|$ (\Prop~\ref{Prop_Bin}):
loosely speaking, we shall observe that most likely $G_1$ is contained in the largest component of $H_3$.
If this is indeed the case, then $\size(H_3)=|\mathcal{F}|+\size(H_2)$, so that
	\begin{eqnarray}
	\size(H_3)=m_{z'}-\mu&\Leftrightarrow&|\mathcal{F}|=m_{z'}-\mu-\size(H_2),\label{eqCquotient2X}\\
	\size(H_3)=m_{z}-\mu&\Leftrightarrow&|\mathcal{F}|=m_{z}-\mu-\size(H_2).\label{eqCquotient3X}
	\end{eqnarray}
Finally, since $\pr\brk{|\mathcal{F}|=m_{z'}-\mu-\size(H_2)}$ is ``close'' to
$\pr\brk{|\mathcal{F}|=m_{z}-\mu-\size(H_2)}$ if $|z-z'|$ is small (by the local limit theorem),
we shall conclude that $P'$ cannot exceed $P$ ``significantly''.

To implement the above sketch,
let $\QQQ$ be the set of all pairs $(\HHH_1,\HHH_2)$ of hypergraphs that satisfy the following three conditions.
\begin{description}
\item[Q1.] $\order(\HHH_2)=\nu$.
\item[Q2.] $\pr\brk{\size(H_3)=m_{z'}-\mu|H_1=\HHH_1,H_2=\HHH_2}\geq n^{-100}$.
\item[Q3.] The largest component of $\HHH_2$ contains the largest component of $\HHH_1$.
\end{description}
The next lemma shows that the processes such that $(H_1,H_2)\in\QQQ$ constitute the dominant contribution.

\begin{lemma}\label{Lemma_CquotientI}
Letting $P''=\pr\brk{\size(H_3)=m_{z'}-\mu\wedge (H_1,H_2)\in\QQQ}$, we have $P'\leq P''+n^{-99}$.
\end{lemma}
\begin{proof}
Let $\mathcal{R}$ signify the set of all pairs $(\HHH_1,\HHH_2)$ such that {\bf Q1} is satisfied.
Since $H_3=H_d(n,p_{z'})$, we have
	$P'=\pr\brk{\size(H_3)=m_{z'}-\mu\wedge(H_1,H_2)\in\mathcal{R}}.$
Therefore, letting $\bar\QQQ_2$ (resp.\ $\bar\QQQ_3$)
denote the set of all $(\HHH_1,\HHH_2)\in\mathcal{R}$ that violate {\bf Q2} (resp.\ {\bf Q3}),
we have
	\begin{eqnarray}
	P'-P''&\leq&\pr\brk{\size(H_3)=m_{z'}-\mu\wedge (H_1,H_2)\in\mathcal{R}\setminus\QQQ}\nonumber\\
		&\leq&\pr\brk{\size(H_3)=m_{z'}-\mu|(H_1,H_2)\in\bar\QQQ_2}
			+\pr\brk{(H_1,H_2)\in\bar\QQQ_3}\nonumber\\
		&\stacksign{\bf Q2}{\leq}&n^{-100}+\pr\brk{(H_1,H_2)\in\bar\QQQ_3}.
			\label{eqCquotient4X}
	\end{eqnarray}
Furthermore, if $(H_1,H_2)\in\bar\QQQ_3$, then either $H_1$ does not feature a component of order $\Omega(n)$,
or $H_2$ has two such components.
Since $\bink{n-1}{d-1}q_1>(d-1)^{-1}+\eps$ due to our choice of $\eps>0$, \Thm~\ref{Thm_global} entails
that the probability of either event is $\leq n^{-100}$.
Thus, the assertion follows from~(\ref{eqCquotient4X}).
\qed\end{proof}

Finally, we can compare $P$ and $P''$ as follows.

\begin{lemma}\label{Lemma_CquotientII}
We have $P''\leq(1+\alpha/3)P$.
\end{lemma}
\begin{proof}
Consider $(\HHH_1,\HHH_2)\in\QQQ$ and let us condition on the event $(H_1,H_2)=(\HHH_1,\HHH_2)$.
Let $\Delta=m_z-\mu-\size(H_2)$, $\Delta'=m_z'-\mu-\size(H_2)$.
We claim that
	\begin{equation}\label{eqCquotient5X}
	\abs{\brk{\bink{\nu}d-\size(H_1)}q_2-\Delta'}\leq n^{0.51};
	\end{equation}
for if $\abs{\brk{\bink{\nu}d-\size(H_1)}q_2-\Delta'}>n^{0.51}$, then
the Chernoff bound~(\ref{eqChernoff}) entails that
	\begin{eqnarray*}
	\pr\brk{\size(H_3)=m_{z'}-\mu|(H_1,H_2)=(\HHH_1,\HHH_2)}
		&\stacksign{(\ref{eqCquotient2X})}{=}&
			\pr\brk{|\mathcal{F}|=\Delta'|(H_1,H_2)=(\HHH_1,\HHH_2)}\\
		&\stacksign{(\ref{eqCquotient1X})}{\leq}&\exp\brk{-n^{0.01}}<n^{-100},
	\end{eqnarray*}
in contradiction to {\bf Q2}.
Thus, if $|z-z'|<\beta$ for a small enough $\beta>0$, then \Prop~\ref{Prop_Bin} yields
	\begin{equation}\label{eqCquotient6X}
	\pr\brk{|\mathcal{F}|=\Delta'|(H_1,H_2)=(\HHH_1,\HHH_2)}\leq
		(1+\alpha/3)\pr\brk{|\mathcal{F}|=\Delta|(H_1,H_2)=(\HHH_1,\HHH_2)},
	\end{equation}
because $|\Delta'-\Delta|=|z'-z|\sigma$, and $\Var(|\mathcal{F}|)\sim\bink{\nu}dq_2=\Omega(\sigma^2)$.
Since~(\ref{eqCquotient6X}) holds for all $(\HHH_1,\HHH_2)\in\QQQ$, the assertion follows.
\qed\end{proof}
Finally, \Lem~\ref{Lemma_Cquotient} is an immediate consequence of \Lem s~\ref{Lemma_CquotientI} and~\ref{Lemma_CquotientII}.

\subsection{Proof of \Lem~\ref{Lemma_fg}}\label{Sec_fg}

Set $m_-=m_0-z^*\sigma$, $m_+=m_0+z^*\sigma$,
and let
	\begin{eqnarray*}
	P(m)=n\pr\brk{\order(\hnm)=\nu\wedge\bar\size(\hnm)=\bar\mu},&&
	B_z(m)=\pr\brk{\Bin\bc{\bink{n}d,p_z}=m}.
	\end{eqnarray*}
Then for all $z\in\brk{-z^*,z^*}$ we have
	\begin{eqnarray*}\nonumber
	f(z)&=&\sum_{m=0}^{\bink{n}d}P(m)B_z(m)\leq n\cdot\pr\brk{\Bin\bc{\bink{n}d,p_z}\not\in\brk{m_-,m_+}}+
		\sum_{m_-\leq m\leq m_+}P(m)B_z(m)\\
		&\stacksign{(\ref{eqChernoff})}{\leq}&n^{-100}+\sum_{m_-\leq m\leq m_+}P(m)B_z(m).
	\end{eqnarray*}
because $0\leq P(m)\leq n$.
Hence,
	\begin{equation}\label{eqfg1}
	f(z)=O(n^{-100})+\sum_{m_-\leq m\leq m_+}P(m)B_z(m).
	\end{equation}

Now, to approximate the sum on the r.h.s.\ of~(\ref{eqfg1}) by the convolution $g*\phi(z)$,
we replace the sum by an integral.
To this end, we decompose the interval $J=\brk{m_-,m_+}$ into $k$ subsequent pieces $J_1,\ldots,J_k$
of lengths in-between $\frac{\sigma}{2\log n}$ and $\frac{\sigma}{\log n}$.
Then \Lem~\ref{Lemma_gcont} entails that
	\begin{equation}\label{eqfg2}
	P(m)=(1+o(1))P(m')+O(n^{-20})
	\qquad\mbox{for all $m,m'\in J_i$ and all $1\leq i\leq k$}.
	\end{equation}
Moreover, \Prop~\ref{Prop_Bin} yields that 
	\begin{equation}\label{eqfg3}
	B_z(m)\sim\frac1{\sqrt{2\pi}\sigma}\exp\bc{-\frac{(m-m_z)^2}{2\sigma^2}}
		\qquad\mbox{for all $m,m'\in J_i$ and all $1\leq i\leq k$}.
	\end{equation}
Further, let $I_i=\{\sigma^{-1}(x-m_0):x\in J_i\}$ and set $M_i=\min J_i\cap\ZZ$.
Combining~(\ref{eqfg2}) and~(\ref{eqfg3}), we obtain
	\begin{eqnarray}
	\sum_{m\in J_i}P(m)B_z(m)&=&
		O(n^{-18})+(1+o(1))P(M_i)\sum_{m\in J_i}B_z(m)\nonumber\\
		&=&(1+o(1))P(M_i)\int_{I_i}\phi(\zeta-z)d\zeta+O(n^{-18})\nonumber\\
		&=&(1+o(1))\int_{I_i}P(m_{\zeta})\phi(\zeta-z)d\zeta+O(n^{-18}).\label{eqfg4}
			\end{eqnarray}
As $|\zeta|\leq z^*$ for all $\zeta\in I_i$, we have $P(m_{\zeta})=g(\zeta)$.
Therefore, (\ref{eqfg4}) yields
	\begin{equation}\label{eqfg5}
	\sum_{m\in J_i}P(m)B_z(m)=(1+o(1))\int_{I_i}g(\zeta)\phi(\zeta-z)d\zeta+O(n^{-18}).
	\end{equation}
Summing~(\ref{eqfg5}) for $i=1,\ldots,k$, we get
	\begin{eqnarray}\nonumber
	f(z)&\stacksign{(\ref{eqfg1})}{=}&
		O(n^{-18})+(1+o(1))\sum_{i=1}^k\int_{I_i}g(\zeta)\phi(\zeta-z)d\zeta\\
		&=&O(n^{-18})+(1+o(1))\int_{-z^*}^{z^*}g(\zeta)\phi(\zeta-z)d\zeta.
		\label{eqfg6}
	\end{eqnarray}
As $f(\zeta)=g(\zeta)=0$ if $|\zeta|>z^*$,
the assertion follows from~(\ref{eqfg6}).

\subsection{Proof of \Lem~\ref{Lemma_fh}}\label{Sec_fh}

\begin{lemma}\label{Lemma_chibound}
We have $\chi<1$.
\end{lemma}
\begin{proof}
We can write the function $F(z)$ from \Lem~\ref{Lemma_fexplicit} as
$F(z)=\xi_1\exp(-\frac{\chi(z-\xi_2)^2}{2})$ with suitable coefficients $\xi_1,\xi_2$.
Hence, the variance of the probability distribution $\|F\|_1^{-1}F$ is $\chi^{-1}$.
To bound this from below, note that $\|F-g*\phi\|_1=o(1)$ by \Lem~\ref{Lemma_fg}.
Moreover, as the convolution of two probability measures is a probability measure, we have $\|g\|_1\sim\|F\|_1$.
Therefore,
	\begin{equation}\label{eqchibound}
	\chi^{-1}=\Var(\|F\|_1^{-1}F)\sim\Var(\|g\|_1^{-1}(g*\phi))=\Var(\|g\|_1^{-1}g)+1.
	\end{equation}
Finally, \Lem~\ref{Lemma_gcont} implies that $\Var(\|g\|_1^{-1}g)>0$, and thus the assertion follows from~(\ref{eqchibound}).
\qed\end{proof}

Now, we shall see that $h*\phi=F$, where $F$ is the function from \Lem~\ref{Lemma_fexplicit}.
Then the assertion follows directly from \Lem~\ref{Lemma_fexplicit}.
To compute $h*\phi$, let
	\begin{eqnarray*}
	\eta_1&=&\frac{n}{2\pi\rho^{d/2}\sqrt{1-\chi}\sigma_\order\sigma}\exp\bc{-\frac{\chi\theta-\kappa^2}{2\chi}},\
	\eta_2=-\kappa/\chi,\ \eta_3=\chi^{-1}-1,\mbox{ and }\eta_4=\eta_1\sqrt{2\pi \eta_3};
	\end{eqnarray*}
note that the definition of $\eta_4$ is sound due to \Lem~\ref{Lemma_chibound}.
Then 
$h(z)=\eta_4\phi_{\eta_2,\eta_3}$.
Hence, $h*\phi=\eta_4\phi_{\eta_2,\eta_3+1}$.
Finally, an elementary but tedious computation shows that $\eta_4\phi_{\eta_2,\eta_3+1}=F$.

\subsection{Proof of \Thm~\ref{Thm_Hnplocal}}\label{Sec_Hnplocal}

Suppose that $\nu=(1-\rho)n+x$ and $\mu=(1-\rho^d)m_0+y$, where $n^{-\frac12}\bink{x}y\in\III$.
Let $\alpha>0$ be arbitrarily small but fixed, and let $\Gamma=\Gamma(\alpha)>0$ be a sufficiently large number.
Moreover, set 
$\mathcal{P}=\pr\brk{\order(\hnp)=\nu\wedge\size(\hnp)=\mu},$ and let
	$$\mathcal{B}(m)=\pr\brk{\Bin\bc{\bink{n}d,p}=m},\ 
		\mathcal{Q}(m)=\pr\brk{\order(\hnm)=\nu\wedge\size(\hnm)=\mu}.$$
Then, letting $m$ range over non-negative integers, we define
	$$
	S_1=\hspace{-3mm}\sum_{m:|m-m_0|\leq\Gamma\sigma}\hspace{-3mm}\mathcal{B}(m)\mathcal{Q}(m),\ 
	S_2=\hspace{-3mm}\sum_{m:\Gamma\sigma<|m-m_0|\leq L\sqrt{n}}\hspace{-3mm}\mathcal{B}(m)\mathcal{Q}(m),\
	S_3=\hspace{-3mm}\sum_{m:|m-m_0|>L\sqrt{n}}\hspace{-3mm}\mathcal{B}(m)\mathcal{Q}(m),
	$$
so that we can rewrite~(\ref{eqRueckschritt0}) as
	\begin{equation}\label{eqRueckschritt1}
	\mathcal{P}=S_1+S_2+S_3.
	\end{equation}
We shall estimate the three summands $S_1,S_2,S_3$ separately.

Let us first deal with $S_3$.
As$\bink{n}dp=O(n)$, the Chernoff bound~(\ref{eqChernoff}) entails that
	$\sum_{m:|m-m_0|>L\sqrt{n}}\mathcal{B}(m)\leq n^{-2}.$
Since, $0\leq\mathcal{Q}(m)\leq 1$, this implies
	\begin{equation}\label{eqRueckschritt2}
	S_3\leq n^{-2}.
	\end{equation}

To bound $S_2$, we need the following lemma.
\begin{lemma}\label{Lemma_RueckschrittAux}
There is a constant $K'>0$ such that
	$\mathcal{Q}(m)\leq K'n^{-1}$ for all $m$ such that $|m-m_0|\leq L\sqrt{n}$.
\end{lemma}
\begin{proof}
Let $z=\sigma^{-1}(m-m_0)$, so that $m=m_z$.
Then $|z|=O(L)$, because $\sigma=\Omega(\sqrt{n})$.
In addition, let $\bar\mu_m=m-\mu$, so that
	\begin{equation}\label{eqRueckschritt3}
	\mathcal{Q}(m)=
		\pr\brk{\order(\hnm)=\nu\wedge\bar\size(\hnm)=\bar\mu_m}
		= g_{n,\nu,\bar\mu_m}(z)/n.
	\end{equation}
Let $c_z=dm_z/n=dm/n$.
Then by \Lem~\ref{Lemma_rhoz}, the solution $0<\rho_z<1$ to the equation
$\rho_z=\exp(c_z(\rho_z^{d-1}-1))$ satisfies $|\rho_z-\rho|=O(zn^{-\frac12})$.
Therefore, we have $|\nu-(1-\rho_z)n|,|\bar\mu_m-(1-\rho_z^d)m|\leq\sqrt{n}\ln n$.
Hence, combining the first part of \Lem~\ref{Lemma_fexplicit} with \Cor~\ref{Lemma_FourierAux1},
we conclude that $g_{n,\nu,\bar\mu_m}(z)\leq K\gamma_0$.
Thus, the assertion follows from~(\ref{eqRueckschritt3}).
\qed\end{proof}
Choosing $\Gamma>0$ large enough, we can achieve that $\sum_{m:|m-m_0|> \Gamma\sigma}\mathcal{B}(m)\leq\alpha/K'$.
Therefore, \Lem~\ref{Lemma_RueckschrittAux} entails that
	\begin{equation}\label{eqRueckschritt4}
	S_2=\sum_{m: \Gamma\sigma<|m-m_0|\leq L\sqrt{n}}\mathcal{B}(m)\mathcal{Q}(m)\leq\alpha n^{-1}.
	\end{equation}

Concerning $S_1$, we employ \Prop~\ref{Prop_Bin} to obtain
	\begin{eqnarray}\label{eqRueckschritt5}
	\mathcal{B}(m)&\sim&\frac1{\sqrt{2\pi}\sigma}\exp\brk{-\frac{(m-m_0)^2}{2\sigma^2}}
		\qquad\mbox{ if }|m-m_0|\leq\Gamma\sigma.
	\end{eqnarray}
In addition, let $0<\rho_m<1$ signify the unique number such that $\rho_m=\exp(\frac{dm}n(\rho_m^{d-1}-1))$.
Then \Lem~\ref{Lemma_rhoz} yields
	$\rho_m=\rho+\Delta_m/n+o(n^{-1/2})$,
where $\Delta_m=-\frac{m-m_0}{\sigma}\cdot\sigma_\order\lambda$.
Hence, $1-\rho_m^d=1-\rho^d-\Xi_m/m+o(n^{-1/2})$, where $\Xi_m=\frac{dm_0}n\Delta_m\rho^{d-1}$.
Thus, \Thm~\ref{Thm_Hnmlocal} entails that $\mathcal{Q}(m)\sim\varphi(m)$, where
	\begin{eqnarray}\label{eqRueckschritt6}
	\varphi(m)&=&
		\inv{2\pi\sqrt{\tsigma_\order^2\tsigma_\size^2-\tsigma_{\order\size}^2}}\\
		&&\;\times
			\exp\brk{-\frac{\tsigma_\order^2\tsigma_\size^2}{2(\tsigma_\order^2\tsigma_\size^2-\tsigma_{\order\size}^2)}
			\left(\frac{(x+\Delta_m)^2}{\tsigma_\order^2}-\frac{2\tsigma_{\order\size}(x+\Delta_m)(y+\Xi_m)}
				{\tsigma_\order^2\tsigma_\size^2}+\frac{(y+\Xi_m)^2}{\tsigma_\size^2}\right)}.
	\nonumber
	\end{eqnarray}
Now, combining~(\ref{eqRueckschritt5}) and~(\ref{eqRueckschritt6}), we can approximate the sum $S_1$ by an integral as follows:
	\begin{eqnarray}\label{eqRueckschritt7}
	S_1&\sim&\sum_{m:|m-m_0|\leq\Gamma\sigma}\frac1{\sqrt{2\pi}\sigma}\exp\brk{-\frac{(m-m_0)^2}{2\sigma^2}}\varphi(m)
		\sim\int_{-\Gamma}^{\Gamma}\varphi(m_z)\phi(z)dz.
	\end{eqnarray}
Further, since $\Delta_{m_z}=-z\sigma_{\order}\lambda=-z\Theta(\sqrt{n})$ and $\Xi_{m_z}=-zd\sigma_{\order}\lambda m_0\rho^{d-1}/n=-z\Theta(\sqrt{n})$,
and because $\tau_\order,\tau_\size,\tau_{\order \size}=\Theta(\sqrt{n})$, the function $\varphi(m_z)$ decays exponentially
as $z\rightarrow\infty$.
Therefore, choosing $\Gamma$ large enough, we can achieve that
	\begin{equation}\label{eqRueckschritt8}
	\int_{\RR\setminus\brk{-\Gamma,\Gamma}}\varphi(m_z)\phi(z)dz<\alpha/n.
	\end{equation}

Combining~(\ref{eqRueckschritt1}), (\ref{eqRueckschritt2}), (\ref{eqRueckschritt4}), (\ref{eqRueckschritt7}), and~(\ref{eqRueckschritt8}), 
we obtain $|\mathcal{P}-\int_{-\infty}^{\infty}\varphi(m_z)\phi(z)dz|\leq3\alpha/n$.
Finally, a trite computation shows that the integral $\int_{-\infty}^{\infty}\varphi(m_z)\phi(z)dz$ equals
the expression $P(x,y)$ from \Thm~\ref{Thm_Hnplocal}.

\section{The Probability that $\hnm$ is Connected: Proof of \Thm~\ref{thm:cnm}}\label{sect:cnm}

In this section we follow the way paved in~\cite{CMS04} to derive the probability that $H_d(\nu,\mu)$
is connected (\Thm~\ref{thm:cnm}) from the local limit theorem for $\hnp$ (\Thm~\ref{Thm_Hnplocal}).
Let $\JJJ\subset(d(d-1)^{-1},\infty)$ be a compact interval, and let $\mu(\nu)$ be a sequence such that
$\zeta=\zeta(\nu)=d\mu/\nu\in\JJJ$ for all $\nu$.
The basic idea is to choose $n$ and $p$ in such a way that $|\nu-\Erw(\order(\hnp))|,|\mu-\Erw(\size(\hnp))|$ are ``small'',
i.e., $\nu$ and $\mu$ will be ``probable'' outcomes of $\order(\hnp)$ and $\size(\hnp)$.
Since given that $\order(\hnp)=\nu$ and $\size(\hnp)=\mu$, the largest component of $\hnp$ is a uniformly distributed
connected graph of order $\nu$ and size $\mu$, we can then express the probability that $H_d(\nu,\mu)$
is connected in terms of the probability $\chi=\pr\brk{\order(\hnp)=\nu\wedge \size(\hnp)=\mu}$. 
More precisely, one can derive from \Thm~\ref{Thm_global} that
	\begin{equation}\label{eqchi}
	\chi\sim\bink{n}{\nu}\bink{\bink{\nu}d}\mu c_d(\nu,\mu) p^{\mu}(1-p)^{\bink{n}{d}-\bink{n-\nu}d-\bink{\nu}d+\mu},
	\end{equation}
where the expression on the right hand side equals the \emph{expected} number of components of order $\nu$ and size $\mu$
occurring in $\hnp$.
Then, one solve~(\ref{eqchi}) to obtain an explicit expression for $c_d(\nu,\mu)$ in terms of $\chi$.
The (somewhat technical) details of approach were carried out in~\cite{CMS04}, where the following lemma was established.

\begin{lemma}\label{Lemma_CnmAux}
Suppose that $\nu>\nu_0$ for some large enough number $\nu_0=\nu_0(\JJJ)$.
Then there exist an integer $n=n(\nu)=\Theta(\nu)$ and a number $0<p=p(\nu)<1$ such that the following is true.
\begin{enumerate}
\item Let $c=\bink{n-1}{d-1}p$.
	Then $(d-1)^{-1}<c=O(1)$, and letting $0<\rho=\rho(c)<1$ signify the solution to~(\ref{eqCOMV}), we have
	$\nu=(1-\rho)n$ and $|\mu-(1-\rho^d)\bink{n}dp|=O(1)$.
\item The solution $r$ to~(\ref{eq:hnp}) satisfies $|r-\rho|=o(1)$.
\item Furthermore,
		$c_d(\nu,\mu)\sim\chi^{-1}\cdot nuvw\cdot\Phi^{\nu}$,
	where $\Phi=(1-r)^{1-\zeta} \,r^{r/(1-r)} \,\bc{1-r^d}^{\zeta/d},$
		\begin{eqnarray}\label{eqAminCnm1}
		u&=&2\pi\sqrt{r(1-r)(1-r^d)c/d},\\
		v&=&\exp\brk{\frac{(d-1)r c}{2(1-r)}\big(1-2r^{d-1}+r^{d-2}\big)},\mbox{ and}\\
		w&=&\left\{\begin{array}{cl}
			\exp\brk{\frac{c^2}{2d}(1-r^d)\cdot\frac{1-r^d-(1-r)^d}{(1-r)^d}}=\frac{c^2r(1+r)}{2}&\mbox{ if $d=2$},\\
			1&\mbox{ if }d>2.\end{array}\right.
		\label{eqAminCnm2}
		\end{eqnarray}
\end{enumerate}
\end{lemma}

Now, \Thm~\ref{Thm_Hnplocal} yields the asymptotics
	$\chi\sim(2\pi)^{-1}\brk{\sigma_\order^2\sigma_\size^2-\sigma_{\order\size}^2}^{\frac12},$
where 
	\begin{eqnarray}\label{eqAminCnm3}
	\sigma_\order^2&=&\frac{\rho\left(1-\rho + c(d-1)(\rho-\rho^{d-1})\right)}{(1-c(d-1)\rho^{d-1})^2}n,\\
	\sigma_\size^2&=&c^2\rho^d\frac{2 + c(d-1)\rho^{2d-2} -2c(d-1)\rho^{d-1} +c(d-1)\rho^d - \rho^{d-1}-\rho^d}{(1-c(d-1)\rho^{d-1})^2}n + (1-\rho^d)\frac{cn}{d},\label{eqAminCnm4}\\
	\sigma_{\order\size}&=&c\rho\frac{1-\rho^d - c(d-1)\rho^{d-1}(1-\rho)}{(1-c(d-1)\rho^{d-1})^2}n.\label{eqAminCnm5}
	\end{eqnarray}
Further, since $r\sim\rho$, $n=\nu/(1-\rho)$, and $c\sim\frac{1-r}{1-r^d}\zeta$, we can express
(\ref{eqAminCnm1})--(\ref{eqAminCnm5}) solely in terms of $\nu$, $r$, and $\zeta$.
As $c_d(\nu,\mu)\sim\chi\cdot nuvw\cdot\Phi^{\nu}$ by \Lem~\ref{Lemma_CnmAux}, we thus
obtain an explicit formula for $c_d(\nu,\mu)$ in terms of $\nu$, $r$, and $\zeta$.
Finally, simplifying this formula via elementary manipulations, we obtain the expressions stated in \Thm~\ref{sect:cnm}.

\begin{remark}
While Lemma~\ref{Lemma_CnmAux} was established in Coja-Oghlan, Moore, and Sanwalani~\cite{CMS04}, the exact limiting distribution of $\order,\size(\hnp)$
was not known at that point.
Therefore, Coja-Oghlan, Moore, and Sanwalani could only compute the $c_d(\nu,\mu)$ up to a constant factor.
By contrast, combining \Thm~\ref{Thm_Hnplocal} with \Lem~\ref{Lemma_CnmAux},
here we have obtained \emph{tight} asymptotics for $c_d(\nu,\mu)$.
\end{remark}

\section{The Probability that $\hnp$ is Connected: Proof of \Thm~\ref{thm:cd}}\label{sect:proofc}

Let $\JJJ\subset(0,\infty)$ be a compact set, and let
$0<p=p(\nu)<1$ be a sequence such that $\zeta=\zeta(\nu)=\bink{\nu-1}{d-1}p\in\JJJ$ for all $\nu$.
To compute the probability $c_d(\nu,p)$ that a random hypergraph $H_d(\nu,p)$ is connected, we will establish that
	\begin{equation}\label{eqHnp1}
	\pr\brk{\cN(\hnp) = \nu} \sim \bink{n}{\nu} c_d(\nu,p)(1-p)^{\bink{n}{d}-\bink{n-\nu}{d}-\bink{\nu}{d}}
	\end{equation}
for a suitably chosen integer $n>\nu$.
Then, we will employ \Thm~\ref{Thm_Nlocal} (the local limit theorem for $\order(\hnp)$) to compute the l.h.s.\ of (\ref{eqHnp1}), so that we can just
solve~(\ref{eqHnp1}) for $c_d(\nu,p)$.

We pick $n$ as follows.
By \Thm~\ref{Thm_global} for each integer $N$ such that
$\bink{N-1}{d-1}p>(d-1)^{-1}$ the transcendental equation $\rho(N)=\exp(\bink{N-1}{d-1}p(\rho(N)^{d-1}-1))$ has
a unique solution $\rho(N)$ that lies strictly between $0$ and $1$.
We let $n=\max\{N:(1-\rho(N))n'<\nu\}$.
Moreover, set $\rho=\rho(n)$ and $c=\bink{n-1}{d-1}p$, and let $0<s<1$ be such that $(1-s)n=\nu$.
Then
	\begin{equation}\label{eqHnp2}
	|n -(1-\rho)\nu|< O(1).
	\end{equation}

To establish (\ref{eqHnp1}), note that the r.h.s.\ is just the expected number of components of order $\nu$ in $\hnp$.
For there are $\bink{n}{\nu}$ ways to choose the vertex set $\comp$ of such a component, and
the probability that $\comp$ spans a connected hypergraph is $c_d(\nu,p)$.
Moreover, if $\comp$ is a component, then $\hnp$ features no edge that connects $\comp$ with $V\setminus\comp$,
and there are $\bink{n}{d}-\bink{n-\nu}{d}-\bink{\nu}{d}$ possible edges of this type, each being present with probability $p$ independently.
Hence, we conclude that
	\begin{equation}\label{eqHnp3}
	\pr\brk{\cN(\hnp) = \nu} \leq \bink{n}{\nu} c_d(\nu,p)(1-p)^{\bink{n}{d}-\bink{n-\nu}{d}-\bink{\nu}{d}}.
	\end{equation}	
On the other hand,
	\begin{equation}\label{eqHnp4}
	\pr\brk{\cN(\hnp) = \nu} \leq \bink{n}{\nu} c_d(\nu,p)(1-p)^{\bink{n}{d}-\bink{n-\nu}{d}-\bink{\nu}{d}}
		\pr\brk{\order(H_d(n-\nu,p))<\nu},
	\end{equation}	
because the r.h.s.\ equals the probability that $\hnp$ has \emph{exactly} one component of order $\nu$.
Furthermore, as $|\nu -(1-\rho)n|< O(1)$ by (\ref{eqHnp2}), \Thm~\ref{Thm_global} entails that $\pr\brk{\order(H_d(n-\nu,p))<\nu}\sim1$.
Hence, combining~(\ref{eqHnp3}) and~(\ref{eqHnp4}), we obtain~(\ref{eqHnp1}).

To derive an explicit formula for $c_d(\nu,p)$ from~(\ref{eqHnp1}), we need the following lemma.

\begin{lemma}\label{Lemma_HnpAux}
\begin{enumerate}
\item We have $c=\zeta(1-s)^{1-d}\bc{1+\bink{d}2\frac{s}{(1-s)n}+O(n^{-2})}$.
\item The transcendental equation~(\ref{eq:varrho}) has a unique solution $0<\varrho<1$, which satisfies $|s-\varrho|=O(n^{-1})$.
\item Letting $\Psi(x) = (1-x)x^\frac{x}{1-x}\exp\bc{\frac{\zeta}{d}\cdot\frac{1-x^d-(1-x)^d}{(1-x)^d}}$, we have
	$\Psi(\varrho)^\nu \sim \Psi(s)^\nu$.
\end{enumerate}
\end{lemma}
\begin{proof}
Regarding the first assertion, we note that
\begin{eqnarray}
\frac{(1-s)^{d-1}\bink{n-1}{d-1}}{\bink{(1-s)n-1}{d-1}}
 &=&\prod_{j=1}^{d-1}1+\frac{sj}{(1-s)n-j}
 =1+\bink{d}2\frac{s}{1-s}+O(n^{-2}).\label{eqHnpAux1}
\end{eqnarray}
Since $\bink{\nu-1}{d-1}c=\zeta\binnd$ and $\nu=(1-s)n$, (\ref{eqHnpAux1}) implies 1.

With respect to 2., set
	$$\varphi_z:(0,1)\rightarrow\RR,\ t\mapsto\exp\bc{z\frac{t^{d-1}-1}{(1-t)^{d-1}}}\mbox{ for }z>0.$$
Then $\lim_{t\searrow0}\varphi_z(t)=\exp(-z)>0$, while $\lim_{t\nearrow1}\varphi_z(t)=0$.
In addition, $\varphi_z$ is convex for any $z>0$.
Therefore, for each $z>0$ there is a unique $0<t_z<1$ such that $t_z=\varphi_z(t_z)$, whence (\ref{eq:varrho}) has the
unique solution $0<\varrho=t_{\zeta}<1$.
Moreover, letting $\zeta'=(1-\rho)^{d-1}c$, we have $\rho=t_{\zeta'}$.
Thus, since $t\mapsto t_z$ is differentiable by the implicit function theorem and
$|\zeta-\zeta'|=O(n^{-1})$ by 1., we conclude that $|\varrho-\rho|=O(n^{-1})$.
In addition, $|s-\rho|=O(n^{-1})$ by (\ref{eqHnp2}).
Hence, $|s-\varrho|=O(n^{-1})$, as desired.

To establish the third assertion, we compute
\begin{eqnarray}\nonumber
\frac{\partial}{\partial x}\Psi(x) &=&
	(1-x)^{-d-1}x^\frac{2x-1}{1-x}\exp\bc{\frac{\zeta}{d}\frac{1-x^d-(1-x)^d}{(1-x)^d}}\\
	&&\qquad\qquad\qquad\times\bc{c(1-x)(x-x^d)+(1-x)^d x\ln x}.\label{eqDiffPsi}
\end{eqnarray}
As $\varrho=\exp\bc{\zeta\frac{\varrho^{d-1}-1}{(1-\varrho)^{d-1}}}$, (\ref{eqDiffPsi}) entails that
$\frac{\partial}{\partial x}\Psi(\varrho)=0$. 
Therefore, Taylor's formula yields that
$\Psi(s)-\Psi(\varrho) = O(s-\varrho)^2=O(n^{-2}),$
because $s-\varrho=O(n^{-1})$ by the second assertion.
Consequently, we obtain
	$$\bcfr{\Psi(s)}{\Psi(\varrho)}^n=\bc{1+\frac{\Psi(s)-\Psi(\varrho)}{\Psi(\varrho)}}^n
			\sim\exp\bc{n\cdot \frac{\Psi(s)-\Psi(\varrho)}{\Psi(\varrho)}}
			=\exp(O(n^{-1}))\sim 1,$$
thereby completing the proof of 3.
\qed\end{proof}

\noindent\emph{Proof of \Thm~\ref{thm:cd}.}
Since $|\nu-(1-\rho)n|< O(1)$ by~(\ref{eqHnp2}), 
\Thm~\ref{Thm_Nlocal} yields that $\pr\brk{\order(\hnp)=\nu}\sim(2\pi)^{-\frac12}\sigma_\order^{-1}$,
where $\sigma_\order$ is given by~(\ref{eq:defsigmaN}).
Plugging this formula into~(\ref{eqHnp1}) and estimating the binomial coefficient $\bink{n}{\nu}$ via Stirling's formula, we obtain
\begin{eqnarray}\label{eq:cd}
c_d(\nu,p)&\sim&s^{s n}(1-s)^{(1-s)n} (1-p)^{\bink{n-\nu}{d}+\bink{\nu}{d}-\bink{n}{d}}\cdot u,\mbox{ where}\\
	u^2&=&\frac{(1-s)(1-c(d-1)s^{d-1})^2}{1-s + c(d-1)(s-s^{d-1})}\label{eq:cdu}.
\end{eqnarray}
Let us consider the cases $d=2$ and $d>2$ separately.
\begin{description}
\item[1st case: $d=2$.]
	Since $\nu=(1-s)n$, we get
		\begin{eqnarray*}
		(1-p)^{\bink{n-\nu}d+\bink{\nu}d-\bink{n}d}&=&(1-p)^{s(s-1)n^2}
			\sim\exp\bc{cs(1-s)(n+1)+\frac{c^2}2s(1-s)}.
		\end{eqnarray*}
	Moreover,	(\ref{eq:cdu}) simplifies to $u=1-cs$.
	Hence, using \Lem~\ref{Lemma_HnpAux} and recalling that $n=(1-s)^{-1}\nu$, we can restate~(\ref{eq:cd}) as
		\begin{eqnarray}
		c_2(\nu,p)
			&\sim&\Psi_2(s,\zeta)^{\nu}
					\exp\brk{\frac{\zeta s^2}{1-s}+\zeta s+\frac{ \zeta^2s}{2(1-s)}}
							\brk{1-\zeta\frac{s}{1-s}}\nonumber\\
			&\sim&\Psi_2(\varrho,\zeta)^{\nu}
				\exp\brk{\frac{\zeta \varrho^2}{1-\varrho}
+\zeta \varrho+\frac{ \zeta^2\varrho}{2(1-\varrho)}}
							\brk{1-\zeta\frac{\varrho}{1-\varrho}}\nonumber\\
			&=&\Psi_2(\varrho,\zeta)^{\nu}
				\exp\brk{\frac{\zeta \varrho (2+\zeta)}{2(1-\varrho)}}
							\brk{1-\zeta\frac{\varrho}{1-\varrho}}
						\label{eqHnp5}
		\end{eqnarray}
	Finally, for $d=2$ the unique solution to~(\ref{eq:varrho}) is just $\varrho=\exp(-\zeta)$.
	Plugging this into~(\ref{eqHnp5}), we obtain the formula stated in \Thm~\ref{thm:cd}.
\item[2nd case: $d>2$.]
	We have $\bink{n}dp^2=o(1)$, because $\bink{n-1}{d-1}p=c=\Theta(1)$.
	Hence, as $\nu=(1-s)n$, we get
		\begin{eqnarray*}
		v&=&(1-p)^{\bink{\nu}d+\bink{n-\nu}{d}-\bink{n}d}
			\sim\exp\brk{\bc{p+\frac{p^2}2}\bc{\bink{n}d-\bink{\nu}d-\bink{n-\nu}{d}}}\\
			&\sim&\exp\bc{\frac{cn}d(1-s^d-(1-s)^d)+\frac{c(d-1)}2((1-s){s}^{d-1}+s(1-s)^{d-1})}.
		\end{eqnarray*}
	Plugging this into~(\ref{eq:cd}) and invoking \Lem~\ref{Lemma_HnpAux}, we obtain
		\begin{eqnarray*}
		c_d(\nu,p)&\sim&{s}^{sn}(1-s)^{(1-s)n}uv\\
						&\sim&\Psi_d(\varrho,\zeta)^{\nu}\exp\brk{\frac{\zeta(d-1)\varrho(1-\varrho^d-(1-\varrho)^d)}{2(1-\varrho)^d}+\frac{\zeta(d-1)\varrho}2\bc{\bcfr{\varrho}{1-\varrho}^{d-2}+1}}\\
			&&\qquad\qquad\times\brk{1-\zeta(d-1)\bcfr{\varrho}{1-\varrho}^{d-1}}\brk{1+\zeta(d-1)(\varrho-\varrho^{d-1})(1-\varrho)^{-d}}^{-\frac12},
		\end{eqnarray*}
	which is exactly the formula stated in \Thm~\ref{thm:cd}.
\end{description}
\qed

\section{The Conditional Edge Distribution: Proof of \Thm~\ref{Thm_edges}}\label{sect:edgedist}

Let $\JJJ\subset(0,\infty)$ and $\III\subset\RR$ be compact sets, and let $0<p=p(\nu)<1$ be a sequence such that
$\zeta=\zeta(\nu)=\bink{\nu-1}{d-1}p\in\JJJ$ for all $\nu$.
To compute the limiting distribution of the number of edges of $H_d(\nu,p)$ given that this random hypergraph is connected,
we choose $n>\nu$ as in \Sec~\ref{sect:proofc}.
Thus, letting $c=\bink{n-1}{d-1}p$, we know from \Lem~\ref{Lemma_HnpAux}  that $c>(d-1)^{-1}$, and that
the solution $0<\rho<1$ to (\ref{eqCOMV}) satisfies $(1-\rho)n\leq\nu\leq(1-\rho)n+1$.
Now, we investigate the random hypergraph $\hnp$ \emph{given that } $\order(\hnp)=\nu$.
Then the largest component of $\hnp$ is a random hypergraph $H_d(\nu,p)$ \emph{given that $H_d(\nu,p)$ is conncected}.
Therefore,
	\begin{eqnarray}\nonumber
	\pr\brk{|E(H_d(\nu,p)|=\mu|H_d(\nu,p)\mbox{ is connected}}
		&=&\pr\brk{\size(\hnp)=\mu|\order(\hnp)=\nu}\\
		&=&\frac{\pr\brk{\size(\hnp)=\mu\wedge\order(\hnp)=\nu}}{\pr\brk{\order(\hnp)=\nu}}
		\label{eqCond1}
	\end{eqnarray}
Furthermore, as $|\nu-(1-\rho)n|< O(1)$, we can
apply \Thm~\ref{Thm_Hnplocal} to get an explicit expression for the r.h.s.\ of~(\ref{eqCond1}).
Namely, for any integer $\mu$ such that $y=n^{-\frac12}(\mu-(1-\rho^d)\bink{n}dp)\in\III$ we obtain
	\begin{eqnarray}\nonumber
	\pr\brk{|E(H_d(\nu,p)|=\mu|H_d(\nu,p)\mbox{ is connected}}\\
		&\hspace{-6cm}\sim&\;\hspace{-3cm}
		\bcfr{\sigma_\order^2}{2\pi(\sigma_\order^2\sigma_\size^2-\sigma_{\order\size}^2)}^{\frac12}
			\exp\bc{-\frac{\sigma_\order^2y^2}{2(\sigma_\order^2\sigma_\size^2-\sigma_{\order\size}^2)}},
	\label{eqCond2}
	\end{eqnarray}
where 	\begin{eqnarray}\label{eqCond3}
	\sigma_\order^2&=&\frac{\rho\left(1-\rho + c(d-1)(\rho-\rho^{d-1})\right)}{(1-c(d-1)\rho^{d-1})^2}n,\\
	\sigma_\size^2&=&c^2\rho^d\frac{2 + c(d-1)\rho^{2d-2} -2c(d-1)\rho^{d-1} +c(d-1)\rho^d - \rho^{d-1}-\rho^d}{(1-c(d-1)\rho^{d-1})^2}n + (1-\rho^d)\frac{cn}{d},\label{eqCond4}\\
	\sigma_{\order\size}&=&c\rho\frac{1-\rho^d - c(d-1)\rho^{d-1}(1-\rho)}{(1-c(d-1)\rho^{d-1})^2}n.\label{eqCond5}
	\end{eqnarray}

Thus, we have derived a formula for $\pr\brk{|E(H_d(\nu,p)|=\mu|H_d(\nu,p)\mbox{ is connected}}$
in terms of $n$, $c$, and $\rho$.
In order to obtain a formula in terms of $\nu$, $\zeta$, and the solution $\varrho$ to (\ref{eq:varrho}), we just
observe that $|c-\zeta(1-\rho)^{1-d}|=O(n^{-1})$ and $|\rho-\varrho|=O(n^{-1})$ by \Lem~\ref{Lemma_HnpAux}, and that
$|n-(1-\rho)^{-1}\nu|=O(n^{-1})$.
Finally, substituting $\varrho$ for $\rho$, $\zeta(1-\varrho)^{1-d}$ for $c$, and $(1-\varrho)^{-1}\nu$ for $n$
in (\ref{eqCond3})--(\ref{eqCond5}) and plugging the resulting expressions into~(\ref{eqCond2}) yields the formula for
$\pr\brk{|E(H_d(\nu,p)|=\mu|H_d(\nu,p)\mbox{ is connected}}$ stated in \Thm~\ref{thm:edgedist}.

\subsubsection{Acknowledgment.}
We thank Johannes Michali\v{c}ek for helpful discussions on the use of Fourier analysis for proving \Thm~\ref{Thm_Hnplocal}.

\end{document}